\documentclass[12pt, reqno]{amsart}

\usepackage{fullpage}

\usepackage{graphicx}
\usepackage[vflt]{floatflt}
\usepackage{rotating}
\usepackage{pdflscape}

\usepackage[dvipsnames]{xcolor}
\usepackage{colortbl}

\usepackage{mathtools}
\usepackage{nicematrix}
\usepackage{blkarray}
\usepackage{tikz-cd}
\usepackage[all]{xy}

\usepackage{charter}
\usepackage{newtxmath}
\usepackage{textcomp}

\usepackage{enumitem}
\usepackage{setspace}

\usepackage{todonotes}
\usepackage{comment}

\usepackage{etoolbox}

\usepackage[hypertexnames=false,
    colorlinks,
    citecolor=red,
    linkcolor=blue,
    urlcolor=red,
    bookmarksopen,
    backref=page]{hyperref}

\usepackage[depth=3]{bookmark}

\let\oldtocsection=\tocsection
\let\oldtocsubsection=\tocsubsection
\let\oldtocsubsubsection=\tocsubsubsection
\renewcommand{\tocsection}[2]{%
  \hspace{0em}\vspace{0.5mm}\oldtocsection{#1}{#2}\vspace{0.5mm}}
\renewcommand{\tocsubsection}[2]{%
  \hspace{2em}\vspace{0.25mm}\oldtocsubsection{#1}{#2}\vspace{0.25mm}}
\renewcommand{\tocsubsubsection}[2]{%
  \hspace{2em}\oldtocsubsubsection{#1}{#2}}

\newcounter{dummy}
\makeatletter
\newcommand\myitem[1][]{%
  \item[#1]\refstepcounter{dummy}\def\@currentlabel{#1}}
\makeatother

\numberwithin{equation}{section}
\newtheorem{theorem}{Theorem}[section]
\newtheorem{lemma}[theorem]{Lemma}
\newtheorem{proposition}[theorem]{Proposition}

\newtheorem{corollary}[theorem]{Corollary}
\newtheorem*{thm*}{Theorem}

\theoremstyle{definition}
\newtheorem{example}[theorem]{Example}
\newtheorem{remark}[theorem]{Remark}
\newtheorem{definition}[theorem]{Definition}
\newtheorem{problem}[theorem]{Problem}

\renewcommand{\qedsymbol}{{\vrule height5pt width5pt depth1pt}}

\newtheorem{mainthm}{Theorem}

\newcommand{\be}{\begin{equation}}
\newcommand{\ee}{\end{equation}}
\newcommand{\bes}{\begin{equation*}}
\newcommand{\ees}{\end{equation*}}

\newcommand{\cA}{\mathcal{A}}

\newcommand{\cE}{\mathcal{E}}

\newcommand{\cH}{\mathcal{H}}

\newcommand{\cK}{\mathcal{K}}

\newcommand{\cX}{\mathcal{X}}
\newcommand{\cY}{\mathcal{Y}}

\newcommand{\bB}{\mathbb{B}}
\newcommand{\bC}{\mathbb{C}}
\newcommand{\bD}{\mathbb{D}}

\newcommand{\bF}{\mathbb{F}}
\newcommand{\bH}{\mathbb{H}}
\newcommand{\bL}{\mathbb{L}}
\newcommand{\bM}{\mathbb{M}}
\newcommand{\bN}{\mathbb{N}}

\newcommand{\bT}{\mathbb{T}}
\newcommand{\bU}{\mathbb{U}}
\newcommand{\bZ}{\mathbb{Z}}

\newcommand{\re}{\operatorname{Re}}

\newcommand{\Tr}{\operatorname{Tr}}

\newcommand{\diag}{\operatorname{diag}}

\newcommand{\spn}{\operatorname{span}}

\newcommand{\dom}{\operatorname{dom}}

\newcommand{\fB}{{\mathfrak{B}}}

\newcommand{\fD}{{\mathfrak{D}}}

\newcommand{\fI}{{\mathfrak{I}}}

\newcommand{\fP}{{\mathfrak{P}}}

\newcommand{\fr}{{\mathfrak{r}}}

\newcommand{\sB}{\scr{B}}

\newcommand{\foral}{\text{ for all }}
\newcommand{\qand}{\quad\text{and}\quad}

\newcommand{\AND}{\text{ and }}

\newcommand{\bmat}[1]{\begin{bmatrix} #1 \end{bmatrix}}

\newcommand{\bsmallmat}[1]{\begin{bsmallmatrix} #1 \end{bsmallmatrix}}

\newcommand{\ba}{{\mathbf{a}}}

\def\C{\mathbb{C}}

\def\cK{\mathcal{K}}

\def\scr{\mathscr}

\def\bi{\begin{itemize}}
\def\ei{\end{itemize}}

\def\F{\mathbb{F} _d ^+}

\def\bpm{\begin{pmatrix}}
\def\epm{\end{pmatrix}}
\newcommand{\bsm}{\left[\begin{smallmatrix}}
\newcommand{\esm}{\end{smallmatrix} \right] }

\def\ba{\begin{eqnarray*}}
\def\ea{\end{eqnarray*}}

\def\fskew{\C \ \mathclap{\, <}{\left( \right.} Z \mathclap{  \, \, \, \, \, >}{\left. \right)} \, }

\begin{document}

\nobreakdash

\title{Cyclicity of stable matrix free polynomials over non-commutative operator unit balls}

\author{Jeet Sampat}
\address{Department of Mathematics, University of Manitoba, Winnipeg, Canada}
\email{Jeet.Sampat@umanitoba.ca}

\author{Maximilian Tornes}
\address{Department of Mathematics, University of Manitoba, Winnipeg, Canada}
\email{Maximilian.Tornes@umanitoba.ca}

\date{} %~~~~~~~~ Leave empty unless necessary



\thanks{JS acknowledges partial funding support from the Pacific Institute for the Mathematical Sciences. MT acknowledges partial funding support from the University of Manitoba Graduate Fellowship program.}

\begin{abstract}
    We consider the algebra of square matrices of bounded non-commutative (NC) functions over NC operator unit balls (unit balls corresponding to finite-dimensional operator spaces) and characterize cyclic matrix free polynomials with respect to the canonical weak-* topology. More precisely, we show that a matrix free polynomial generates a weak-* dense left/right ideal if and only if it is stable, i.e., non-singular at every point in the NC operator unit ball. To this end, we establish a version of the Neuwirth--Ginsberg--Newman inequality for stable matrix free polynomials.
    
    We combine our techniques with the theory of realizations to establish cyclicity of stable NC rational functions that are uniformly continuous across the boundary, and we recover known results about cyclicity of NC rational functions in the matrix-valued free Hardy space over the NC unit row-ball. Lastly, we introduce the NC parallel sum function: a stable NC rational function that is contractive over the NC bidisk, which cannot be extended uniformly across the boundary, and determine its cyclicity using properties of accretive operators.
\end{abstract}

\maketitle

\section{Introduction}\label{sec:intro}

Consider the unit disk $\bD$ in the complex plane $\bC$. Neuwirth, Ginsberg and Newman \cite{NGN-approximation} observed that if a polynomial $P \in \bC[z]$ has all its zeros in $\bC \setminus \bD$, then
\begin{equation}\label{eqn:NGN-one-var}
    \left\lvert \frac{P(z)}{P(rz)} \right\rvert \leq 2^{\deg(P)} \foral r < 1 \AND z \in \overline{\bD}.
\end{equation}
In fact, it suffices to prove this for a single factor $z - \alpha$ with $|\alpha| \geq 1$, in which case
\begin{equation*}
    \left\lvert \frac{z - \alpha}{rz - \alpha} \right\rvert = \left\lvert 1 + \frac{(1 - r)z}{rz - \alpha} \right\rvert \leq 1 + \left\lvert \frac{1 - r}{\alpha - rz} \right\rvert \leq 2.
\end{equation*}
It is then straightforward to generalize this to polynomials in several variables that are \emph{stable} (i.e., non-vanishing) with respect to the open unit polydisk $\bD^d$ (for some $d \in \bN$). Indeed, let $P \in \bC[z_1, \dots, z_d]$ be stable on $\bD^d$ and note that $P_w : z \mapsto P(z w_1, \dots, z w_d)$ is stable on $\bD$ for each $w = (w_1, \dots, w_d) \in \overline{\bD^d}$. Using \eqref{eqn:NGN-one-var} for each $P_w$ with $z = 1$ shows at once that
\begin{equation*}
    \left\lvert \frac{P(w)}{P(r w)} \right\rvert \leq 2^{\deg(P)} \foral r < 1 \AND w \in \overline{\bD^d},
\end{equation*}
where $\deg(P)$ denotes the total degree of $P$. We refer to this as the \emph{Neuwirth--Ginsberg--Newman (NGN) inequality}. In this article, we establish an NGN-type inequality for matrices of polynomials in freely non-commuting variables and showcase its application to \emph{cyclicity}.

\subsection{Commutative cyclic polynomials}\label{subsec:intro-comm-cyc.poly} The NGN inequality was used in \cite{NGN-approximation} to study the completeness of dilated systems of functions in $L^p(\bT)$, where $\bT := \partial \bD$ denotes the unit circle. In the context of \emph{Hardy spaces}, this yields the cyclicity of stable polynomials. Recall that the Hardy space $H^2(\bD^d)$ consists of all $f \in \operatorname{Hol}(\bD^d)$ given by a power-series in $d$ (commuting) variables $z = (z_1, \dots, z_d)$, i.e.,
\begin{equation*}
    f(z) = \sum_{\alpha \in \bZ_{\geq 0}^d} c_\alpha z^\alpha,
\end{equation*}
such that $\sum_\alpha|c_\alpha|^2 < \infty$, where each $\alpha = (\alpha_1, \dots, \alpha_d)$ is a $d$-tuple of non-negative integers $\bZ_{\geq 0}$, $c_\alpha \in \bC$, and $z^\alpha := z_1^{\alpha_1} z_2^{\alpha_2} \dots z_d^{\alpha_d}$. Clearly, $H^2(\bD^d)$ is a Hilbert space with respect to the $\ell^2$ inner-product given by the power-series coefficients. We say that $f \in H^2(\bD^d)$ is \emph{cyclic} if
\begin{equation*}
    S[f] := \overline{\spn} \{ z^\alpha f \ : \ \alpha \in \bZ_{\geq 0}^d \} = \overline{\{ Pf \ : \ P \in \bC[z_1, \dots, z_d] \}} = H^2(\bD^d).
\end{equation*}
$H^2(\bD^d)$ is also a reproducing kernel Hilbert space, i.e., the evaluation $\Lambda_w : f \mapsto f(w)$ is a bounded linear functional on $H^2(\bD^d)$ for each $w \in \bD^d$. Thus, a necessary condition for $f$ to be cyclic is that $f$ be non-vanishing on $\bD^d$ (indeed, if $f(w) = 0$ then every $g \in S[f]$ must also satisfy $g(w) = 0$). Combining the NGN inequality with basic facts about $H^2(\bD^d)$ immediately shows that a polynomial $P$ is cyclic if and only if it is stable on $\bD^d$ \cite[Theorem 5]{NGN-approximation} (see also \cite[Theorem 3.4(1)]{Nik-Hardy-sp}).

It is fairly clear how to generalize the notion of cyclicity to more exotic function spaces. Moving beyond the Hardy spaces and the unit polydisk reveals that the zero-set of a polynomial is still the key to determining its cyclicity. However, in most general cases, the set of boundary zeros also comes into play. A recent result of Mironov and the first named author \cite[Corollary 3.2]{MS-joint-cyc} shows that a polynomial is cyclic in a general topological vector space of analytic functions over some open set $\Omega \subset \bC$ if and only if it is stable with respect to the \emph{maximal domain} $\Omega_{\max}$, which consists of all points in $\bC$ (and, therefore, potentially outside of $\Omega$) at which a continuous evaluation functional can be defined on the whole space. Instances of a similar interplay between cyclicity of polynomials and their zero-sets have appeared in a variety of other contexts, such as:
\begin{enumerate}
    \item Dirichlet-type spaces \cite{BKKLSS-Dirichlet, Bergqvist-Dirichlet} and anisotropic Dirichlet-type spaces \cite{KKRS-aniso-Dirichlet} on $\bD^d$,
    \item weighted analytic $\ell^p$ spaces, i.e., $\ell^p_A(w)$ \cite{ST-Lpa-cyclicity},
    \item Dirichlet-type spaces on the \emph{Euclidean unit ball} $\bB_d \subset \bC^d$ \cite{KV-Dirichlet-ball, VZ-Dirichlet-ball},
    \item Drury--Arveson space and weighted Besov spaces on $\bB_d$ \cite{APRSS-Drury-Arveson}, etc.
\end{enumerate}
While progress has been significant in specific instances, the characterization of cyclicity in general analytic function spaces across one and several variables remains a deep and challenging open problem. It nevertheless continues to be a promising and productive area of study \cite{Sam-cyc-survey}.

\subsection{Free analysis}\label{subsec:intro-free.analysis}

\emph{Free analysis} or \emph{non-commutative function theory} is a conceptual generalization of the classical commutative function theory, and can be traced back to the work of Taylor from the 1970s \cite{Taylor-NC-functional-calc, Taylor-NC-functions}. While Taylor's original work on the extension of functional calculus to non-commuting tuples went largely unnoticed, NC function theory has since blossomed into a successful theory with works spanning across several areas of independent interests \cite{AMY-NC-book, BB-NC-book, KVV-NC-book}, and has applications to operator theory, systems/control theory as well as free probability \cite{BGM-NC-appl, HV-NC-appl, Pop-NC-appl, Voi-NC-appl}.

\subsubsection*{\texorpdfstring{\textbf{NC notation}}{NC notation}} The $d$-dimensional \emph{non-commutative (NC) universe} is defined as
\begin{equation*}
    \bM^d := \bigsqcup_{n \in \bN} M_n^d \cong \bigsqcup_{n \in \bN} M_n(\bC^d) \cong \bigsqcup_{n \in \bN} \bC^d \otimes M_n,
\end{equation*}
where $M_n^d$ consists of all $d$-tuples of $n \times n$ (complex) matrices for each $n \in \bN$ and $M_n := M_n^1$. We endow each $M_n^d$ with the \emph{supremum norm}
\begin{equation*}
    \|X\| := \max_{1 \leq j \leq d} \|X_j\| \ \foral \ X = (X_1, \dots, X_d) \in M_n^d,
\end{equation*}
and use it to obtain the \emph{disjoint union topology} on $\bM^d$: $\Omega \subseteq \bM^d$ is open in the disjoint union topology if and only if $\Omega(n) := \Omega \cap M_n^d$ is open for each $n \in \bN$. Although we do not make significant use of it in the sequel, one often considers the \emph{uniform topology} as well, which is generated by the following basic open sets: given any $X \in M_n^d$ and $r > 0$, the \emph{uniform NC ball} (with center $X$ and radius $r$) is defined as
\begin{equation*}
    B_\infty(X,r) := \bigsqcup_{m = 1}^\infty \{ Y \in M_{mn}^d \ : \ \|Y - I_m \otimes X \| < r \}.
\end{equation*}

To imbue further communication between each \emph{level} $\Omega(n)$, one usually works with \emph{NC sets}. Here, $\Omega \subset \bM^d$ is an NC set if $X, Y \in \Omega \Rightarrow X \oplus Y := \bsmallmat{X & 0 \\ 0 & Y} \in \Omega$. Given an NC set $\Omega$ as above, we say that $F : \Omega \to \bM^1$ is an \emph{NC function} if
\begin{enumerate}
    \item $F$ is \emph{graded}: $X \in \Omega(n) \Rightarrow F(X) \in M_n$,
    \item $F$ \emph{respects direct sums}: $X, Y \in \Omega(n) \Rightarrow F(X \oplus Y) = F(X) \oplus F(Y)$,
    \item $F$ \emph{respects similarities}: $X \in \Omega(n)$, $S \in GL_n$ and $S^{-1}XS = (S^{-1}X_1 S, \dots, S^{-1} X_d S) \in \Omega(n) \Rightarrow F(S^{-1} X S) = S^{-1} F(X) S$.
\end{enumerate}
The most fundamental feature of free analysis is that a mild local boundedness condition on an NC function $F$ is guaranteed to ensure that $F$ is continuous and holomorphic, in the sense that given any $X \in \Omega(n)$ and any ``direction" $H \in M_n^d$, the directional or G\^ateaux derivative of $F$ at $X$ in the direction $H$ exists:
\begin{equation}\label{eqn:NC-Gateaux-deriv}
    \partial_H F(X) := \lim_{t \to 0} \frac{F(X + tH) - F(X)}{t}.
\end{equation}
Moreover, $F$ has a total or Fr\'echet derivative at each $X \in \Omega$ (see \cite[Chapter 7]{KVV-NC-book}).

The primary domains of interest for us are the \emph{NC operator (unit) balls}, which are defined as follows. Let $\cE \subseteq \sB(\cH)$ be a $d$-dimensional operator space over some Hilbert space $\cH$, and let $\{Q_1, \dots, Q_d\}$ be a basis for $\cE$. We then introduce the linear operator-valued polynomial $Q(Z) := \sum_j Q_j Z_j$, which is to be interpreted functionally on $\bM^d$ via
\begin{equation*}
    Q(X) := \sum_{j = 1}^d Q_j \otimes X_j \foral X \in \bM^d.
\end{equation*}
The NC operator ball $\bD_Q$ corresponding to this linear map $Q$ is given by
\begin{equation*}
    \bD_Q := \{ X \in \bM^d \ : \ \|Q(X)\| < 1 \},
\end{equation*}
where $\|Q(X)\|$ is the operator norm of $Q(X)$ in $\sB(\cH \otimes \bC^n)$ if $X \in M_n^d$. It is easy to check that each $\bD_Q$ is a bounded NC set that is also uniformly open and \emph{matrix convex} \cite[Proposition 2.6]{Sampat-Shalit-Weak-star}. When $\cH \cong \bC^l$ is finite-dimensional, we may view $Q = (Q_{ij})_{l \times l}$ as simply a matrix of linear polynomials $\{Q_{ij}\}$ in $d$ freely non-commuting variables.

\begin{example}\label{eg:NC.row.ball.polydisk}
    Let $\cH = \bC^d$ for some $d \in \bN$.
    \begin{enumerate}
        \item The \emph{NC unit row-ball} $\fB_d$ is defined as
        \begin{equation*}
            \fB_d := \left\{ X \in \bM^d \ : \ \left\| \sum_{j = 1}^d X_jX_j^* \right\| < 1 \right\}.
        \end{equation*}
        However, it is easy to see that this corresponds to the linear operator-valued map
        \begin{equation*}
            Q(Z) = \bmat{Z_1 & \dots & Z_d} \cong \sum_{j = 1}^d E_{1j} Z_j,
        \end{equation*}
        where the $E_{ij}$'s denote the standard matrix units.
        \item The \emph{NC unit polydisk} $\fD_d$ is defined as
        \begin{equation*}
            \fD_d := \{ X \in \bM^d \ : \ \|X\| < 1 \},
        \end{equation*}
        which clearly corresponds to
        \begin{equation*}
            Q(Z) = \diag(Z_1, \dots, Z_d) \cong \sum_{j = 1}^d E_{jj} Z_j.
        \end{equation*}
    \end{enumerate}
\end{example}

\subsubsection*{\texorpdfstring{\textbf{NC function algebras}}{NC function algebras}} Given an NC operator ball $\bD_Q$, we consider the algebra of bounded NC functions on $\bD_Q$, i.e.,
\begin{equation*}
    H^\infty(\bD_Q) := \left\{ F : \bD_Q \to \bM^1 \ : \ F \text{ is NC} \AND \|F\|_Q := \sup_{X \in \bD_Q} \|F(X)\| < \infty \right\}.
\end{equation*}
These algebras are referred to as the \emph{NC Schur--Agler class}, and have appeared most prominently in the context of the NC interpolation problem and transfer function realizations, the isomorphism problem, as well as NC spectral radius formulae \cite{BMV-interpolation, KS-zeros-TFR, Sampat-Shalit-iso-prob, Sampat-Shalit-Weak-star, Shalit-Shamovich-spec-rad}.

It is clear that $H^\infty(\bD_Q)$ is a Banach algebra, but it is also an operator algebra: for each $k \in \bN$, consider the algebra $M_k(H^\infty(\bD_Q))$ of $k \times k$ matrices of $H^\infty(\bD_Q)$-functions endowed with the matrix norm
\begin{equation*}
    \|(F_{ij})\|^{(k)}_Q := \sup_{X \in \bD_Q} \|(F_{ij}(X))\|,
\end{equation*}
and note that the family of matrix norms $\{\|\cdot\|_Q^{(k)}\}_{k \in \bN}$ satisfies the Blecher--Ruan--Sinclair axioms \cite[Theorem 2.3.2]{BLM-op-sp-book}. It is also clear that the algebra $\bC \langle Z \rangle$ of polynomials in $d$ freely non-commuting variables, i.e., \emph{free polynomials}, sits inside $H^\infty(\bD_Q)$. We will also consider the algebra $M_k(\bC \langle Z \rangle)$, which consists of $k \times k$ matrices of free polynomials. Any element $P \in M_k(\bC \langle Z \rangle)$ is then called a \emph{matrix free polynomial}.

Based on the discussion surrounding \eqref{eqn:NC-Gateaux-deriv}, we note that each $F \in H^\infty(\bD_Q)$ is uniformly holomorphic, and exhibits an NC power-series centered at $0$ (see \cite[Theorem 7.21]{KVV-NC-book}):
\begin{equation*}
    F(Z) = \sum_{\alpha \in \bF_d^+} c_\alpha Z^\alpha,
\end{equation*}
where $\bF_d^+$ is the free monoid generated by the \emph{alphabet} $\cA = \{1, \dots, d\}$, each $\alpha = \alpha_1 \dots \alpha_l$ is a \emph{word} in $\cA$, each $c_\alpha \in \bC$ is a scalar, and $Z^\alpha = Z_{\alpha_1} \dots Z_{\alpha_l}$ is an \emph{NC monomial}. Moreover, this series converges absolutely and uniformly on $r \bD_Q$ for all $r < 1$.

\subsection{Main results}\label{subsec:intro-main.results} Fix $d, k \in \bN$ and let $\bD_Q \subset \bM^d$ be an NC operator ball as above. Also, let $GL_k(\bC \langle Z \rangle)$ be the collection of all \emph{invertible} matrix free polynomials, i.e., $P \in M_k(\bC \langle Z \rangle)$ such that $P^{-1} \in M_k(\bC \langle Z \rangle)$.

\subsubsection*{\texorpdfstring{\textbf{Section \ref{sec:Free.NGN.inequality}}}{Section 3}} We begin with an exploration of NGN-type bounds for matrix free polynomials in $M_k(\bC \langle Z \rangle)$ viewed as elements of $M_k(H^\infty(\bD_Q))$. $P \in M_k(\bC \langle Z \rangle)$ is said to be \emph{$Q$-stable} (or \emph{$\bD_Q$-stable}) if $\det P(X) \neq 0$ for all $X \in \bD_Q$. For any $r < 1$, we introduce the map $P^{(r)} : X \mapsto P(rX)$ and, in the spirit of the classical NGN inequality, want to know if
\begin{equation*}
    \| P^{(r)}(X)^{-1} P(X) \| \qand \| P(X) P^{(r)}(X)^{-1} \|
\end{equation*}
are uniformly bounded (in $r \in (0,1)$ and $X \in \bD_Q$) for any given $Q$-stable $P \in M_k(\bC \langle Z \rangle)$. Unfortunately, this is not true even in very simple cases. For instance, we show in Example \ref{example:Jordan.block.pencil} that the above quantities are unbounded for the case $d = 1$ and the $2 \times 2$ matrix polynomial $L(Z) := \bsmallmat{1 - Z & -Z \\ 0 & I - Z}$, which is clearly $\fD_1$-stable. Part of the reason why this is the case is that $L$ is not `irreducible,' in a sense. Indeed, we note in Remark \ref{rem:pencil.atomic.factor} that $L$ has a factorization
\begin{equation*}
    L(Z) = \bmat{I & 0 \\ 0 & I - Z} \bmat{I & -Z \\ 0 & I} \bmat{I - Z & 0 \\ 0 & I}.
\end{equation*}
However, we note that each of these factors satisfy an NGN-type bound. Our first main observation is that this idea follows through in general, and we need to consider the \emph{atomic factorization} of matrix free polynomials. To this end, we utilize tools such as \emph{linearization} of matrix free polynomials, and also the notion of spectral radius formulae associated to NC operator balls as recently introduced by Shalit and Shamovich \cite{Shalit-Shamovich-spec-rad}. These concepts are summarized briefly in Section \ref{sec:prelims} for the uninitiated reader.

Following the notation from \cite{Cohn-fir-local-book, Helton-Klep-Volcic-Free-factor}, recall that $M_k(\bC \langle Z \rangle)$ is a \emph{semifir}, i.e., a semi-free ideal ring, for each $k \in \bN$. Among other things, the semifir property ensures that every non-zero divisor $P \in M_k(\bC \langle Z \rangle)$ exhibits a factorization $P = P_1 \dots P_l$, where each $P_j$ is an \emph{atom}, i.e., it is not a product of two $F, G \not\in GL_k(\bC \langle Z \rangle)$. Moreover, each $P_j$ is unique up to stable associativity (see Definition \ref{def:stable.assoc}). With this in mind, our first main result is the following NGN-type result for $Q$-stable atoms.

\begin{mainthm}\label{mainthm:NGN.Q-stable.atom}
    If $P \in M_k(\bC \langle Z \rangle)$ is a $Q$-stable atom, then 
     \begin{align*}
        \sup_{r < 1} \big\| \big( P^{(r)} \big)^{-1} P \big\|_Q < \infty \qand \sup_{r < 1} \big\| P \big(P^{(r)} \big)^{-1} \big\|_Q < \infty.
    \end{align*}
\end{mainthm}

An immediate consequence of the atomic factorization and Theorem \ref{mainthm:NGN.Q-stable.atom} is Corollary \ref{cor:NGN.general}, which states that if $P$ is $Q$-stable and has atomic factorization $P = P_1 \dots P_l$ then
\begin{align*}
        \sup_{r < 1} \big\| \big( {P_1}^{(r)} \big)^{-1} P_1 \big( {P_2}^{(r)} \big)^{-1} P_2 \dots \big( {P_l}^{(r)} \big)^{-1} P_l \big\|_Q &< \infty, \\
        \sup_{r < 1} \big\| P_1 \big( {P_1}^{(r)} \big)^{-1} P_2 \big( {P_2}^{(r)} \big)^{-1} \dots P_l \big( {P_l}^{(r)} \big)^{-1} \big\|_Q &< \infty.
\end{align*}
Of course, one wonders what the optimal bounds are in either case. The proof of Theorem \ref{mainthm:NGN.Q-stable.atom} reveals that our bounds depend on various quantities associated with $P$ and not just its degree. We provide several examples and showcase concrete situations where a reasonable bound can be obtained (see Proposition \ref{prop:improved.bounds} and Example \ref{eg:concrete.bounds}). In general, however, it is unclear if there is an optimized strategy to achieve the best possible bound (see also Section \ref{subsec:NGN.using.FM.realzn}, where we incorporate the theory of \emph{FM-realizations} to obtain NGN-type bounds).

\subsubsection*{\texorpdfstring{\textbf{Section \ref{sec:cyclicity}}}{Section 4}} In a recent work of Shalit and the first named author \cite{Sampat-Shalit-Weak-star}, it was established that $H^\infty(\bD_Q)$ has a canonical weak-* topology. Moreover, the canonical pre-dual is unique in the sense that evaluations at matrix points $X \in \bD_Q$ are all weak-* continuous. Using basic facts from functional analysis, we show in Proposition \ref{prop:uniq.pre-dual.matrix.case} that this gives rise to a canonical weak-* topology on $M_k(H^\infty(\bD_Q))$ for each $k \in \bN$, given by a unique pre-dual (in the same sense as above). 

Given any $F \in M_k(H^\infty(\bD_Q))$, we say that $F$ is \emph{left/right (weak-*) cyclic} if the weak-* closed left/right ideal generated by $F$ is equal to $M_k(H^\infty(\bD_Q))$. As in the classical case, we note in Lemma \ref{lem:cyclic.iff.identity} that $F$ is left/right cyclic if and only if the `constant' NC function $I$ lies in its left/right weak-* closed ideal, from which we obtain a simple necessary condition for cyclicity: if $F$ is left/right cyclic then $F$ is $Q$-stable. Using Theorem \ref{mainthm:NGN.Q-stable.atom} we obtain our main result in Section \ref{sec:cyclicity}, which shows that the converse holds above when $F$ is a matrix free polynomial.

\begin{mainthm}\label{mainthm:cyc.matrix.free.poly}
    $P \in M_k(\bC \langle Z \rangle)$ is left/right cyclic in $M_k(H^\infty(\bD_Q))$ if and only if $P$ is $Q$-stable.
\end{mainthm}

\subsubsection*{\texorpdfstring{\textbf{Section \ref{sec:cyclicity.rational}}}{Section 5}} Following \cite{HMV-NC.rat}, recall that an NC rational function $\fr \in \fskew$ is an equivalence class of NC rational expressions in the freely non-commuting variables $Z = (Z_1, \dots, Z_d)$, and its domain (denoted by $\dom \fr$) is the union of the domains of each of these NC rational expressions. In Section \ref{sec:cyclicity.rational}, we use the theory of \emph{descriptor realizations} and extend the above theorem to certain matrices of \emph{NC rational functions} in $M_k(H^\infty(\bD_Q))$.

\begin{mainthm}\label{mainthm:cyc.NC.rationals}
     Let $\fr \in M_k(\fskew)$ be a matrix of NC rational functions and let $s \bD_Q\subset \dom \fr$ for some $s>1$. Then, $\fr$ is left/right cyclic in $M_k(H^\infty(\bD_Q))$ if  and only if $\fr$ is $Q$-stable.
\end{mainthm}

In Corollary \ref{cor:cyclic.NC.rat.Hardy}, we use the well-known identification of $H^\infty(\fB_d)$ as the \emph{multiplier algebra} of the \emph{free Hardy space} $\bH^2_d$ (see \cite[Theorem 3.1]{SSS-algebras}) to recover a known complete characterization of cyclic NC rational functions in $\bH^2_d$. Indeed, we know from \cite[Theorems A and C]{JMS-ratFock} that $\fr \in \bH^2_d$ if and only if there exists $s > 1$ such that $s \fB_d \subset \dom \fr$, and that $\fr$ is cyclic in $\bH^2_d$ if and only if $\fr$ is $\fB_d$-stable. Recent work of Arora, Augat, Jury and Sargent \cite{Arora-Augat-Jury-Sargent-free-OPA} and of Arora \cite{AroraPhD} employs techniques from the theory of \emph{optimal polynomial approximations (OPAs)} to provide different proofs of the same fact. Similarly, our Theorems \ref{mainthm:cyc.matrix.free.poly} and \ref{mainthm:cyc.NC.rationals} provide yet another approach to establish the cyclicity of stable NC rational functions in $\bH^2_d$. We note that this is possible since the weak operator topology on $H^\infty(\fB_d)$ coincides with the weak-* topology as above \cite{Davidson-Pitts-inv}, however, such a representation is not known to exist for a general $\bD_Q$ (e.g., \cite[Theorem 2.4]{Sampat-Shalit-iso-prob} demonstrates certain challenges when $\bD_Q = \fD_d$).

Lastly, in Section \ref{subsec:NC.parallel.sum}, we introduce the \emph{NC parallel sum function}
\begin{equation*}
    \fP(Z,W) := (I - Z)(2I - Z - W)^{-1}(I - W) = \big( (I - Z)^{-1} + (I - W)^{-1} \big)^{-1} = (I - W)(2I - Z - W)^{-1}(I - Z),
\end{equation*}
which is inspired by the parallel sum operation on positive semi-definite matrices as introduced by Anderson and Duffin \cite{AD-parallel-sum}. In Corollary \ref{cor:cyc.of.NC.par.sum}, we show that $\fP$ is a $\fD_2$-stable contractive NC map that (i) has a singularity at $(I,I) \in \overline{\fD_2}$ and (ii) is \emph{accretive}, i.e., $\re \fP(Z,W) \succeq 0$ for all $(Z,W) \in \fD_2$. In Corollary \ref{cor:cyclicity.of.accretive.NC.func}, we show that any accretive NC function $F \in M_k(H^\infty(\bD_Q))$ is always left/right cyclic. In particular, the cyclicity of $\fP$ follows from its accretivity instead of Theorem \ref{mainthm:cyc.NC.rationals} or even a direct application of our free NGN-type bound (see Remark \ref{rem:failure.NGN.par.sum.cyc}).

\section{Preliminaries}\label{sec:prelims}

\subsection{Spectral radius corresponding to an operator space}\label{subsec:spec.rad}

Given any operator space $\cE$, we consider the unit ball of all square matrices over $\cE$, i.e.,
\begin{equation*}
    \bB_\cE := \bigsqcup_{n \in \bN} \{ X \in M_n(\cE) \ : \ \|X\|_n < 1 \},
\end{equation*}
where $\{\left\| \cdot \right\|_n\}_{n \in \bN}$ forms a compatible family of matrix-norms on $\cE$. In what follows, let $\cE \subseteq \sB(\cH)$ be a concrete operator space for some Hilbert space $\cH$ and suppose $\dim \cE = d \in \bN$. We also fix a basis $\{Q_1, \dots, Q_d\}$ for $\cE$ and introduce the linear operator-valued polynomial $Q(Z) := \sum_j Q_j Z_j$. The corresponding NC operator ball $\bD_Q$ then provides a coordinate representation for $\bB_\cE$ via
\begin{equation*}
    \bD_Q(n) \ni X \ \leftrightarrow \ Q(X) \in M_n(\cE)
\end{equation*}
for all $n \in \bN$. The closure of $\bD_Q$ in the disjoint union topology is readily verified to be
\begin{equation*}
    \overline{\bD_Q} := \{X \in \bM^d \ : \ \|Q(X)\| \leq 1 \}.
\end{equation*}

Following \cite[Remark 3.6]{Shalit-Shamovich-spec-rad}, we also introduce the \emph{polar dual} $\bD_Q^\circ$ of $\bD_Q$ as the NC set
\begin{equation*}
    \bD_Q^\circ := \left\{ Y \in \bM^d \ : \ \left\| \sum_{j = 1}^d Y_j \otimes X_j \right\| < 1 \ \foral \ X \in \bD_Q \right\}.
\end{equation*}
It is easy to verify that $\bD_Q^\circ$ corresponds to $\bB_{\cE^*}$, i.e., it is the NC operator ball given by the dual operator space $\cE^*$ (as in \cite[Section 2.3]{Pisier_2003}), and $\bD_Q^\circ$ is equal to $\bD_{Q^\circ}$ for some linear operator-valued polynomial $Q^\circ(W) := \sum_j Q^\circ_j Z_j$. Moreover, we have
\begin{equation}\label{eqn:polar.dual.norm.prop}
    \left\| \sum_{j = 1}^d Y_j \otimes X_j \right\| \leq \|Y\|_{Q^\circ} \|X\|_Q \foral Y \in \bD_Q^\circ, \ X \in \bD_Q.
\end{equation}
Lastly, note that $Y \in \overline{\bD_Q^\circ}$ if and only if
\begin{equation}\label{eqn:polar.dual.boundary.condition}
    \left\| \sum_{j = 1}^d Y_j \otimes X_j \right\| \leq 1 \ \foral \ X \in \bD_Q.
\end{equation}

Shalit and Shamovich \cite{Shalit-Shamovich-spec-rad} recently introduced the notion of a spectral radius corresponding to a given finite-dimensional operator space $\cE \subseteq\sB(\cH)$ as above. We make frequent use of some of their main observations in the sequel. Thus, for the reader's convenience, we present the statements of these facts in this subsection. Note that the technical details surrounding the definition are not necessary to digest the rest of our discussion. We nevertheless point the interested reader to \cite[Section 1.5]{Pisier_2003} or \cite[Chapter 17]{Paulsen-Op.Alg.-Book} for background on minimal/spatial and Haagerup tensor products.

\begin{definition}\label{def:NC.spec.rad}
    The \emph{$Q$-spectral radius} of $T = (T_1, \dots, T_d) \in \sB(\cK)^d$ is defined as
    \begin{equation*}
        \rho_Q(T) := \lim_{n \to \infty} \left\| \sum_{|w| = n} T^w \otimes_{\min} Q_{w_1} \otimes_h Q_{w_2} \otimes_h \dots \otimes_h Q_{w_n} \right\|^{\frac{1}{n}}.
    \end{equation*}
    Here, $\cK$ is some Hilbert space, $\otimes_{\min}$ and $\otimes_h$ are the minimal and Haagerup tensor products, respectively, the sum is taken over all words $w = w_1 w_2 \dots w_n \in \bF_d^+$ of length $n$, and the norm is considered inside the operator space
    \begin{equation*}
        \sB(\cK) \otimes_{\min} (\underbrace{\cE \otimes_h \dots \otimes_h \cE}_{n\text{ times}}).
    \end{equation*}
\end{definition}

As noted in \cite[Example 2.10]{Shalit-Shamovich-spec-rad}, the spectral radius corresponding to $\bD_Q = \fB_d$ can simply be written as
\begin{equation*}
    \rho_{\fB_d}(T) = \lim_{n \to \infty} \left\| \sum_{|w| = n} T^w (T^w)^* \right\|^{\frac{1}{2n}},
\end{equation*}
which coincides with the notion of \emph{joint spectral radius} for a tuple of operators as introduced by Bunce \cite{Bunce-spec.rad}, and later explored by Popescu \cite{Popescu-spec.rad}.

The following facts are pertinent to our discussion. Recall that $A = (A_1, \dots, A_d) \in M_k^d$ is said to be \emph{irreducible} if $\{A_1, \dots, A_d\}$ generate $M_k$ as a $\bC$-algebra (equivalently, the $A_j$'s have no common non-trivial invariant subspace).

\begin{theorem}[Corollary 2.12, \cite{Shalit-Shamovich-spec-rad}]\label{thm:spec.rad.properties-SS}
    The following hold for any $A \in M_n^d$ and $ n \in \bN$:
    \begin{enumerate}
        \item $\rho_Q(A) < 1$ if and only if there exists $S \in GL_n$ such that $S^{-1} A S \in \bD_Q$.

        \item If $A$ is irreducible and $\rho_Q(A) = 1$, then there exists $S \in GL_n$ such that $S^{-1} A S \in \overline{\bD_Q}$.
    \end{enumerate}
\end{theorem}

The spectral radius was also shown to be intimately connected to \emph{linear pencils}. Here, for a given $A \in M_n^d$ as above, we define the (monic) linear pencil $L_A$ as the affine map
\begin{equation*}
    L_A(X) := I - \sum_{j = 1}^d A_j \otimes X_j \foral X \in \bM^d,
\end{equation*}
and its \emph{domain of invertibility} is the NC set
\begin{equation*}
    \dom (L_A^{-1}) := \left\{ X \in \bM^d \ : \ \det \left(I - \sum_{j = 1}^d A_j \otimes X_j \right) \neq 0 \right\}.
\end{equation*}

\begin{theorem}[Theorem 3.4, \cite{Shalit-Shamovich-spec-rad}]\label{thm:spec.rad.dom.pencil-SS}
    Let $A \in M_n^d$ for some $ n \in \bN$, and consider an NC operator ball $\bD_Q$ along with its polar dual $\bD_Q^\circ$. Then, we have
    \begin{equation*}
        \rho_{Q^\circ}(A) \leq \frac{1}{r} \Longleftrightarrow r \bD_Q \subset \dom (L_A^{-1}).
    \end{equation*}
    In particular, $\rho_{Q^\circ}(A)<1$ if and only if $L_A^{-1}$ extends uniformly across the boundary of $\bD_Q$.
\end{theorem}

We say that a linear pencil $L_A$ is \emph{irreducible} if the tuple $A$ is irreducible. In essence, irreducible linear pencils play a role that is similar to irreducible polynomials in the classical case. We explore this in the next subsection.

\subsection{Stable associativity and linearization}\label{subsec:stab.lin}

Recall from the introduction that the ring $M_k(\bC \langle Z \rangle)$ consists of all $k \times k$ matrices with entries in $\bC \langle Z \rangle$ for each $k \in \bN$. Also recall that $GL_k(\bC \langle Z \rangle)$ denotes the ring of all invertible matrix free polynomials $F$ such that $F^{-1} \in M_k(\bC \langle Z \rangle)$. We say that $P \in M_k(\bC \langle Z \rangle)$ is \emph{stable} with respect to $\Omega \subseteq \bM^d$ if $\det P(X) \neq 0$ for all $X \in \Omega$. If $\Omega = \bD_Q$ for some $Q$ as before, then we say that $P$ is \emph{$Q$-stable}.

\begin{definition}\label{def:stable.assoc}
    Given any two $P_j \in M_{k_j}(\bC \langle Z \rangle)$, we say that $P_1$ and $P_2$ are \emph{stably associated} if there exist $l_j \in \bN$ with $k_1 + l_1 = k_2 + l_2$, and $F, G \in GL_{k_1 + l_1}(\bC \langle Z \rangle)$ so that
    \begin{equation*}
        \bmat{P_2 & 0 \\ 0 & I_{l_2}} = F \bmat{P_1 & 0 \\ 0 & I_{l_1}} G.
    \end{equation*}
    In this case, we write $P_1 \sim P_2$ to denote stable associativity and note that `$\sim$' is an equivalence relation for non-zero divisor (square) matrices over $\bC \langle Z \rangle$. 
\end{definition}

Fix an NC operator ball $\bD_Q$ and suppose $P \in M_k(\bC \langle Z \rangle)$ is $Q$-stable. Thus, in particular, $P$ is not a zero-divisor in $M_k(\bC \langle Z \rangle)$. A result of Cohn (see \cite[Proposition 3.2.9]{Cohn-fir-local-book}) then guarantees that $P$ admits a factorization $P = P_1 \dots P_l$, where each $P_j$ is an \emph{atom} in $M_k(\bC \langle Z \rangle)$, i.e., it cannot be viewed as the product of two non-invertible elements $F, G \in M_k(\bC \langle Z \rangle)$. Moreover, each $P_j$ is uniquely determined up to stable associativity and it is also $Q$-stable. For the purpose of determining the cyclicity of a $Q$-stable matrix of free polynomials $P \in M_k(\bC \langle Z \rangle)$, it will be sufficient to work with the equivalence class of each of its atomic factors $P_j$, and so we assume, without loss of generality, that $P_j(0) = I$ for each $j$. The only fact we need about atoms is the following result from \cite{Helton-Klep-Volcic-Free-factor}.

\begin{lemma}[Lemma 4.2, \cite{Helton-Klep-Volcic-Free-factor}]\label{lem:Higman's.linearization.trick}
    If $P \in M_k(\bC \langle Z \rangle)$ is such that $P(0) = I$ then $P \sim L_A$ for some linear pencil $L_A$. Moreover, $P$ is an atom if and only if $P \sim L_B$ for some irreducible $L_B$.
\end{lemma}

The proof of this lemma relies on a clever application of \emph{Higman's linearization trick} (see \cite[Theorem 15]{Higman-trick}), which is the following identity for any three matrices $Y_0, Y_1, Y_2$ of compatible sizes:
\begin{equation}\label{eqn:Higman.id}
    \bmat{I & -Y_1 \\ 0 & I} \bmat{Y_0 - Y_1 Y_2 & 0 \\ 0 & I} \bmat{I & 0 \\ -Y_2 & I} = \bmat{Y_0 & -Y_1 \\ -Y_2 & I}.
\end{equation}
Let us illustrate how to achieve linearization with a simple example.

\begin{example}\label{eg:Higman.linearization}
    Let $P(Z,W) = I - \frac{ZW}{2} - \frac{WZ}{2} \in \bC \langle Z, W \rangle$. First, we take $Y_0 = I - \frac{ZW}{2}$, $Y_1 = \frac{W}{\sqrt{2}}$, $Y_2 = \frac{Z}{\sqrt{2}}$ and plug them into \eqref{eqn:Higman.id} to obtain
    \begin{equation*}
        \bmat{I & \frac{-W}{\sqrt{2}} \\ 0 & I} \bmat{P(Z,W) & 0 \\ 0 & I} \bmat{I & 0 \\ \frac{-Z}{\sqrt{2}} & I} = \bmat{I - \frac{ZW}{2} & \frac{-W}{\sqrt{2}} \\ \frac{-Z}{\sqrt{2}} & I}.
    \end{equation*}
    Then, we take $Y_0 = \bsmallmat{I & \frac{-W}{\sqrt{2}} \\ \frac{-Z}{\sqrt{2}} & I}$, $Y_1 = \bsmallmat{\frac{Z}{\sqrt{2}} & 0 \\ 0 & 0}$, $Y_2 = \bsmallmat{\frac{W}{\sqrt{2}} & 0 \\ 0 & 0}$ and plug them into \eqref{eqn:Higman.id} to get
    \begin{equation*}
        \bmat{I & 0 & \frac{-Z}{\sqrt{2}} & 0 \\ 0 & I & 0 & 0 \\ 0 & 0 & I & 0 \\ 0 & 0 & 0 & I} \bmat{I - \frac{ZW}{2} & \frac{-W}{\sqrt{2}} & 0 & 0 \\ \frac{-Z}{\sqrt{2}} & I & 0 & 0 \\ 0 & 0 & I & 0 \\ 0 & 0 & 0 & I} \bmat{I & 0 & 0 & 0 \\ 0 & I & 0 & 0 \\ \frac{-W}{\sqrt{2}} & 0 & I & 0 \\ 0 & 0 & 0 & I} = \bmat{I & -\frac{W}{\sqrt{2}} & -\frac{Z}{\sqrt{2}} & 0 \\ -\frac{Z}{\sqrt{2}} & I & 0 & 0 \\ -\frac{W}{\sqrt{2}} & 0 & I & 0 \\ 0 & 0 & 0 & I}.
    \end{equation*}
    Note that the last row and column are redundant, so further simplifications can be made. It is then straightforward to combine the above two identities and complete the linearization process:
    \begin{equation}\label{eqn:eg.linearztn}
        \bmat{P(Z,W) & 0 \\ 0 & I_2} = \underbrace{\bmat{I & \frac{W}{\sqrt{2}} & \frac{Z}{\sqrt{2}} \\ 0 & I & 0 \\ 0 & 0 & I}}_{=: \ F} \underbrace{\bmat{I & \frac{-W}{\sqrt{2}} & \frac{-Z}{\sqrt{2}} \\ \frac{-Z}{\sqrt{2}} & I & 0 \\ \frac{-W}{\sqrt{2}} & 0 & I}}_{=: \ L_A} \underbrace{\bmat{I & 0 & 0 \\ \frac{Z}{\sqrt{2}} & I & 0 \\ \frac{W}{\sqrt{2}} & 0 & I}}_{=: \ G}.
    \end{equation}

    Lastly, note that $L_A$ corresponds to
    \begin{equation*}
        A = \left(A_Z = \bmat{0 & 0 & \frac{1}{\sqrt{2}} \\ \frac{1}{\sqrt{2}} & 0 & 0 \\ 0 & 0 & 0}, A_W = \bmat{0 & \frac{1}{\sqrt{2}} & 0 \\ 0 & 0 & 0 \\ \frac{1}{\sqrt{2}} & 0 & 0}\right) \in M_3^2,
    \end{equation*}
    which can be readily checked to be irreducible since
    \begin{align*}
        2 \sqrt{2} A_Z A_W^2 = E_{12}; \ 2 \sqrt{2} A_Z^2 A_W = E_{21}; \ 2 A_Z^2 = E_{23}; \ 2 A_W^2 = E_{32}.
    \end{align*}
    All the remaining $E_{ij}$'s can be obtained using the ones above. Thus, $P$ is an atom.
\end{example}

\begin{remark}\label{rem:Higman.linearization.subtlety}
    In the above example, $F, G \in GL_k(\bC \langle Z \rangle)$ have the form $F = I + J_F$ and $G = I + J_G$, where $J_F$ and $J_G$ are strictly upper and lower triangular, respectively. Consequently, $J_F$ and $J_G$ are nilpotent, so it is easy to compute $F^{-1}$ and $G^{-1}$ in this case, which is necessary for the analysis in the next section. This is merely an application of \eqref{eqn:Higman.id} at each step, however, it is important to remark that the resulting linearization for an atom need not be an irreducible pencil $L_A$ in general. It just happened to be the case for $P(Z,W) = 1 -\frac{ZW}{2} - \frac{WZ}{2}$ that we obtained an irreducible pencil straight away.

    As explained in the proof of \cite[Lemma 4.2]{Helton-Klep-Volcic-Free-factor}, in general, one must use Burnside's Theorem \cite[Corollary 5.23]{Bresar-intro-NC-alg} and show that $L_A$ is similar to a linear pencil
    \begin{equation*}
        L_B(Z) := \bmat{L_1(Z) & * & \dots & * \\ 0 & L_2(Z) & \dots & * \\ \vdots & 0 & \ddots & * \\ 0 & 0 & \dots & L_l(Z)}
    \end{equation*}
    for some $l \in \bN$, where each linear pencil $L_j$ is either $I$ or irreducible, and then show that $L_j = I$ for all $j$ except for exactly one $1 \leq j_0 \leq l$. Thus, in the process of showing that a given atom $P$ is stably associated to $L_{j_0}$, one loses the \emph{uni-triangular} structure of $F$ and $G$. It will nevertheless be of interest in the next section to consider the special case where one can obtain such a linearization simply via Higman's trick \eqref{eqn:Higman.id}.
\end{remark}

\begin{definition}\label{def:uni-tri.matrix.polys}
    For each $k \in \bN$, we define $\bU_k(\bC \langle Z \rangle)$ as the collection of all \emph{upper uni-triangular matrix free polynomials}, i.e, $F \in \bU_k(\bC \langle Z \rangle)$ if and only if $F = I + J_F$ for some strictly upper triangular matrix $J_F$. Similarly, define $\bL_k(\bC \langle Z \rangle)$ to be the collection of all \emph{lower uni-triangular matrix free polynomials}.
    
    Given two matrix free polynomials $P_j \in M_{k_j}(\bC \langle Z \rangle)$, we write $P_1\sim_{\bU\bL} P_2$ if there exist $l_j \in \bN$ with $k_1 + l_1 = k_2 + l_2$, and $F\in \bU_{k_1 + l_1}(\bC \langle Z \rangle)$ and $G\in \bL_{k_1 + l_1}(\bC \langle Z \rangle)$ so that
    \begin{equation*}
        \bmat{P_2 & 0 \\ 0 & I_{l_2}} = F \bmat{P_1 & 0 \\ 0 & I_{l_1}} G.
    \end{equation*}
\end{definition}

\section{A free Neuwirth--Ginsberg--Newman-type inequality}\label{sec:Free.NGN.inequality}

Fix a linear operator-valued polynomial $Q$, and let $\bD_Q$ be the corresponding NC operator ball along with its polar dual $\bD_Q^\circ$. For any $F \in H^\infty(\bD_Q)$ and $r < 1$, define
\begin{equation*}
    F^{(r)}(X) := F(rX) \foral X \in \bD_Q.
\end{equation*}
Clearly, $F^{(r)} \in H^\infty(\bD_Q)$ with $\|F^{(r)}\|_Q \leq \|F\|_Q$ for each $r < 1$. As noted in \cite[Proposition 3.5]{Sampat-Shalit-iso-prob}, $F$ exhibits a \emph{homogeneous expansion} $F = \sum_{j \geq 0} F_j$, where each $F_j \in \bC \langle Z \rangle$ satisfies
\begin{equation*}
    F_j(r X) = r^j F_j(X) \foral X \in \bD_Q, \ r < 1.
\end{equation*}
Moreover, it was noted that the map $F \mapsto F_j$ is completely contractive for each $j \geq 0$. An immediate consequence of this observation is the following basic property of matrix free polynomials. Throughout this section, we shall view the ring of matrix free polynomials $M_k(\bC \langle Z \rangle)$ as a subset of $M_k(H^\infty(\bD_Q))$.

\begin{lemma}\label{lem:NGN.polynomials}
    Fix $k \in \bN$ and let $P \in M_k(\bC \langle Z \rangle)$ be given by the homogeneous expansion $P = \sum_{j = 0}^N P_j$ for some $N \geq 0$. Then, we have
    \begin{equation*}
        \sup_{r < 1} \frac{\| P - P^{(r)} \|_Q}{1-r} \leq \sum_{j=1}^N j \|P_j\|_Q.
    \end{equation*}
\end{lemma}

\begin{remark}\label{rem:NGN.polynomials}
     As noted before the statement of the lemma, we know from \cite[Proposition 3.5]{Sampat-Shalit-iso-prob} that $\|P_j\|_Q \leq \|P\|_Q$ for each $0 \leq j \leq N$. Hence,  
     \begin{equation*}
         \sum_{j=1}^N j \|P_j\|_Q\leq \frac{N (N+1)}{2} \|P\|_Q.
     \end{equation*}
     While the RHS above is nicer to keep track of, in practice, the LHS bound is much tighter.
\end{remark}

\begin{proof}
    A straightforward calculation yields
    \begin{equation*}
        \frac{P(Z) - P(rZ)}{1-r}=\sum_{j = 1}^N \frac{1-r^j}{1-r} P_j(Z).
    \end{equation*}
    Therefore, we easily compute
    \begin{equation*}
        \sup_{r < 1} \frac{\|P - P^{(r)} \|_Q}{1-r}\leq \sum_{j=1}^N \sup_{r < 1}\frac{1 - r^j}{1 - r} \|P_j\|_Q \leq \sum_{j=1}^N j \|P_j\|_Q. \qedhere
    \end{equation*}
\end{proof}

We use the previous lemma to obtain a generalization of the NGN inequality for certain matrix free polynomials on $\bD_Q$.

\begin{lemma}\label{lem:NGN.linear.pencil}
    Let $P = F L_A \in M_k(\bC \langle Z \rangle)$ for some $A \in \overline{\bD^\circ_Q}(k)$ and $F\in GL_k(\bC \langle Z \rangle)$ with homogeneous expansion $F = \sum_{j = 0}^N F_j$. Then, 
    \begin{align}
        \sup_{r < 1} \big\| \big(P^{(r)} \big)^{-1} P \big\|_Q &\leq 1 + \big[ (N^2 + N + 1) \|F^{-1}\|_Q \|F\|_Q \big], \label{eqn:NGN.linear.pencil.left.inv}\\
        \sup_{r < 1} \big\| P \big( P^{(r)} \big)^{-1} \big\|_Q &\leq 2 \|F^{-1}\|_Q \|F\|_Q. \label{eqn:NGN.linear.pencil.right.inv}
    \end{align}

    In particular, if $P = L_A$ (i.e., $F = I$), then both constants above become $2$.
\end{lemma}

\begin{remark}\label{rem:NGN.linear.pencil}
    As will be evident from the following proof, similar bounds can be obtained for the case $P = L_A F$. In fact, the bounds for the left-inverse and right-inverse will be swapped in this case.
\end{remark}

\begin{proof}
    Let $r < 1$ and $X \in \bD_Q$ be arbitrary and define $A\otimes X:=\sum_{j=1}^d A_j\otimes X_j$. A quick application of \eqref{eqn:polar.dual.norm.prop} yields $\| r (A \otimes X) \| < 1$, from which it follows that $L_A(rX)$ is invertible. In fact, we may express $L_A(rX)^{-1}$ as a Neumann series:
    \begin{equation*}
        L_A(rX)^{-1} = \sum_{n=0}^\infty r^n (A\otimes X)^n,
    \end{equation*}
    so that
    \begin{equation}\label{eqn:lem.NGN.linear.pencil.proof.1}
        \| L_A(rX)^{-1} \| \leq \sum_{n=0}^\infty r^n = \frac{1}{1-r}.
    \end{equation}
    
    Observe that
    \begin{equation}\label{eqn:lem.NGN.linear.pencil.proof.2}
        P(rX)^{-1} P(X) = I + P(rX)^{-1} (P(X) - P(rX)), 
    \end{equation}
    and
    \begin{equation*}
        P(X) - P(rX) = F(X)-F(rX)+ (rF(rX)-F(X)) (A \otimes X). 
    \end{equation*}
    Hence, we obtain
    \begin{align*}
        \|P(X) - P(rX)\| &\leq \|F(X)-F(rX)\|+ \| F(X)-rF(rX)\| \\
        &= \|F(X) - F(rX)\| + \|F(X) - F(rX) + (1 - r) F(rX) \| \\
        & \leq 2 \|F(X) - F(rX)\| + (1 - r) \|F(rX)\|.
    \end{align*}
    Combining this with \eqref{eqn:lem.NGN.linear.pencil.proof.1}, \eqref{eqn:lem.NGN.linear.pencil.proof.2} and the bound from Lemma \ref{lem:NGN.polynomials} (and Remark \ref{rem:NGN.polynomials}) we have
    \begin{align*}
        \|P(rX)^{-1} P(X)\| &\leq 1 + \|F(rX)^{-1}\| \left(\frac{2 \|F(X)-F(rX)\|}{1-r} + \|F(rX)\| \right) \\
        &\leq 1 + \big[ (N^2 + N + 1) \|F^{-1}\|_Q \|F\|_Q \big].
    \end{align*}
    As $r < 1$ and $X \in \bD_Q$ were arbitrarily chosen, \eqref{eqn:NGN.linear.pencil.left.inv} follows at once from the above bound.

    It remains to show that \eqref{eqn:NGN.linear.pencil.right.inv} holds. To this end, note that
    \begin{align*}
        \|P(X) P(rX)^{-1}\| &= \| F(X) L_A(X) L_A(rX)^{-1} F(rX)^{-1} \| \\
        &\leq \|L_A(X) L_A(rX)^{-1}\| \|F^{-1}\|_Q \|F\|_Q.
    \end{align*}
    As $L_A(X)$ clearly commutes with $L_A(rX)^{-1}$, we can use the bound from \eqref{eqn:NGN.linear.pencil.left.inv} for the case $F = I$ to note that
    \begin{equation*}
        \|L_A(X) L_A(rX)^{-1}\| = \|L_A(rX)^{-1} L_A(X)\| \leq 2.
    \end{equation*}
    This gives us \eqref{eqn:NGN.linear.pencil.right.inv} at once and the proof is complete.
\end{proof}

Recall from Theorem \ref{thm:spec.rad.dom.pencil-SS} that $L_A(X)^{-1}$ is well-defined for each $X \in \bD_Q$ whenever $\rho_{Q^\circ}(A) \leq 1$. The following example shows that Lemma \ref{lem:NGN.linear.pencil} need not hold for all choices of $A$ with $\rho_{Q^\circ}(A) \leq 1$.

\begin{example}\label{example:Jordan.block.pencil}
    Consider the one-variable linear pencil given by $A = \bsmallmat{1 & 1 \\ 0 & 1}$, i.e.,
    \begin{equation*}
        L_A(Z) = \bmat{I - Z & -Z \\ 0 & I - Z}.
    \end{equation*}
    Clearly, $\rho(A) = 1$ and so $L_A(Z)$ is invertible over the NC unit disk $\fD_1$ (using Theorem \ref{thm:spec.rad.dom.pencil-SS}). It is then straightforward to compute for each $r < 1$ and $X \in \fD_1$ that
    \begin{align*}
        L_A(rX)^{-1} L_A(X) &= \bmat{(I - rX)^{-1} & rX(I - rX)^{-2} \\ 0 & (I - rX)^{-1}} \bmat{ I - X & -X \\ 0 & I - X} \\
        &= \bmat{(I - rX)^{-1} (I - X) & (r-1)X(I - rX)^{-2} \\ 0 & (I - rX)^{-1} (I - X)}.
    \end{align*}
    Notice that the top-right entry in the matrix above is unbounded on the set of all $r<1$ and $X \in \fD_1$. To see this, set $X=r$ and let $r\to 1$. Therefore, \eqref{eqn:NGN.linear.pencil.left.inv} does not hold in this case and, by commutativity of $L_A(rX)^{-1}$ and $L_A(X)$, neither does \eqref{eqn:NGN.linear.pencil.right.inv}.
\end{example}

We are now sufficiently prepared to prove Theorem \ref{mainthm:NGN.Q-stable.atom}.

%\begin{theorem}\label{mainthm:NGN.Q-stable.atom}
%    If $P \in M_k(\bC \langle Z \rangle)$ is a $Q$-stable atom, then 
%     \begin{align*}
%        \sup_{r < 1} \big\| \big( P^{(r)} \big)^{-1} P \big\|_Q < \infty \qand \sup_{r < 1} \big\| P \big(P^{(r)} \big)^{-1} \big\|_Q < \infty.
%    \end{align*}
%\end{theorem}

\subsection*{\texorpdfstring{Proof of Theorem \ref{mainthm:NGN.Q-stable.atom}}{Proof of Theorem A}} First, note that
\begin{align*}
    PP(0)^{-1}(P^{(r)}P(0)^{-1})^{-1}&= P(P^{(r)})^{-1},\\
    (P(0)^{-1}P^{(r)})^{-1}P(0)^{-1}P&= (P^{(r)})^{-1}P.
\end{align*}
We therefore assume without loss of generality that $P(0)=I$, by replacing $P$ with $P P(0)^{-1}$ (resp. $P(0)^{-1}P$) if necessary.
    
From Lemma \ref{lem:Higman's.linearization.trick}, we know that $P\sim L_A$ for some irreducible linear pencil $L_A$. Moreover, using the fact that $P$ is $Q$-stable, it is readily checked that $L_A$ is $Q$-stable as well. In particular, $L_A(X)^{-1}$ exists for all $X\in \bD_Q$, from which $\rho_{Q^\circ}(A)\leq 1$ follows via Theorem \ref{thm:spec.rad.dom.pencil-SS}. Next, since $A$ is irreducible, Theorem \ref{thm:spec.rad.properties-SS} shows that $A$ is jointly similar to some $B\in \overline{\bD_Q^\circ}$. Hence, $P\sim L_B$, which means there exist $l \in \bN$ and $F, G \in GL_l(\bC \langle Z \rangle)$ such that 
\begin{equation*}
   \bmat{P(X) & 0 \\ 0 & I} = F(X) \underbrace{\bmat{L_B(X) & 0 \\ 0 & I}}_{=: \ L_{\widetilde{B}}(X)} G(X) \foral X \in \bD_Q.
\end{equation*}
Note that each matrix in the above equation is invertible. Thus,
\begin{equation*}
    \bmat{P(X)^{-1} & 0 \\ 0 & I} = G(X)^{-1} L_{\widetilde{B}}(X)^{-1} F(X)^{-1} \foral X \in \bD_Q.
\end{equation*}

From here, we observe that
\begin{equation*}
    \bmat{P(rX)^{-1}P(X) & 0 \\ 0 & I} = G(rX)^{-1} \ \underbrace{\big( F(rX)L_{\widetilde{B}}(rX) \big)^{-1} F(X) L_{\widetilde{B}}(X)}_{=: \ (P_1^{(r)})^{-1} P_1} \ G(X),
\end{equation*}
and
\begin{equation*}
    \bmat{P(X) P(rX)^{-1} & 0 \\ 0 & I} = F(X) \ \underbrace{L_{\widetilde{B}}(X) G(X) \big(L_{\widetilde{B}}(rX) G(rX)\big)^{-1}}_{=: \ P_2 (P_2^{(r)})^{-1}} \ F(rX)^{-1},
\end{equation*}
for all $r<1$ and $X\in \bD_Q$. If $F = \sum_{j = 0}^N F_j$ and $G = \sum_{j = 0}^M G_j$ are the homogeneous expansions of $F$ and $G$, then since $\widetilde{B} \in \overline{\bD_Q^\circ}$ we can invoke Lemma \ref{lem:NGN.linear.pencil} for $P_1$ and the observation in Remark \ref{rem:NGN.linear.pencil} for $P_2$ to conclude that
\begin{align}
    \begin{split}\label{eqn:lem.NGN.linear.pencil.proof.3}
        \sup_{r < 1} \big\| \big( P^{(r)} \big)^{-1} P \big\|_Q &\leq \Big[1 + \big[ (N^2 + N + 1) \|F^{-1}\|_Q \|F\|_Q \big]\Big] \|G^{-1}\|_Q \|G\|_Q < \infty, \\
        \sup_{r < 1} \big\|P \big( P^{(r)} \big)^{-1} \big\|_Q &\leq \Big[1 + \big[ (M^2 + M + 1) \|G^{-1}\|_Q \|G\|_Q \big]\Big] \|F^{-1}\|_Q \|F\|_Q < \infty.
    \end{split}
\end{align}
This completes the proof. \hfill \qedsymbol

\begin{remark}\label{rem:pencil.atomic.factor}
    In general, Theorem \ref{mainthm:NGN.Q-stable.atom} does not hold if $P \in M_k(\bC \langle Z \rangle)$ is not an atom. Indeed, consider the linear pencil $L_A(Z)\in M_2(\bC \langle Z \rangle)$ introduced in Example \ref{example:Jordan.block.pencil}. As previously explained, the inequality in Theorem \ref{mainthm:NGN.Q-stable.atom} does not hold for $L_A(Z)$. Moreover, it is not an atom:
    \begin{equation*}
        L_A(Z) = \bmat{I & 0 \\ 0 & I - Z} \bmat{I & -Z \\ 0 & I} \bmat{I-Z & 0 \\ 0 & I}.
    \end{equation*}
    Here, note that both $\bsmallmat{I & 0 \\ 0 & I - Z}$ and $\bsmallmat{I - Z & 0 \\ 0 & I}$ are atoms by the virtue of being stably associated to the irreducible pencil $I - Z$, and that $\bsmallmat{I & -Z \\ 0 & I} \in GL_2(\bC \langle Z \rangle)$. We are therefore forced to interpret the NGN inequality for general $Q$-stable matrices of free polynomials as follows.
\end{remark}

\begin{corollary}\label{cor:NGN.general}
    If $P \in M_k(\bC \langle Z \rangle)$ is $Q$-stable with atomic factorization $P = P_1 P_2 \dots P_l$, then
    \begin{align*}
        \sup_{r < 1} \big\| \big( {P_1}^{(r)} \big)^{-1} P_1 \big( {P_2}^{(r)} \big)^{-1} P_2 \dots \big( {P_l}^{(r)} \big)^{-1} P_l \big\|_Q &< \infty, \\
        \sup_{r < 1} \big\| P_1 \big( {P_1}^{(r)} \big)^{-1} P_2 \big( {P_2}^{(r)} \big)^{-1} \dots P_l \big( {P_l}^{(r)} \big)^{-1} \big\|_Q &< \infty.
    \end{align*}
\end{corollary}

\subsection{Further analysis of the bounds}\label{subsec:analysis.of.bounds} In certain situations, one can obtain more concrete bounds than those from the proof of Theorem \ref{mainthm:NGN.Q-stable.atom}. Let $P \in \bC \langle Z \rangle$ be a $Q$-stable atom, given by the homogeneous expansion $P=\sum_{j=0}^N P_j$, and suppose $P\sim_{\bU\bL} L_A$ for some irreducible linear pencil $L_A\in M_n(\bC \langle Z \rangle)$. As we saw in Example \ref{eg:Higman.linearization}, this situation appears naturally at times by simply applying Higman's trick \eqref{eqn:Higman.id} to $P$. Clearly, $L_A$ is $Q$-stable as well. Thus, $L_A(X)^{-1}$ exists for all $X\in\bD_Q$, so that $\rho_{Q^\circ}(A)\leq 1$ by Theorem \ref{thm:spec.rad.dom.pencil-SS}. Therefore, by Theorem \ref{thm:spec.rad.properties-SS}, $A$ is jointly similar to some element $B\in \overline{\bD_{Q^\circ}}$ via a similarity $S$ of appropriate size. Under these assumptions, we can improve the bounds in Theorem \ref{mainthm:NGN.Q-stable.atom} as follows. Recall that the \emph{condition number} of $S$ is defined as the quantity $\kappa(S) := \|S^{-1}\| \|S\|$.

\begin{proposition}\label{prop:improved.bounds}
    If $P \in \bC \langle Z \rangle$, $L_A$, $S$ and $B$ are as above, then
    \begin{align*}
        \sup_{r < 1} \big\| P \big(P^{(r)} \big)^{-1} \big\|_Q &\leq 1+\kappa(S)\sum_{j=1}^N j \|P_j\|_Q,\\
        \sup_{r < 1} \big\| \big( P^{(r)} \big)^{-1} P \big\|_Q &\leq 1+\kappa(S)\sum_{j=1}^N j \|P_j\|_Q.
    \end{align*}
\end{proposition}

\begin{proof}
    Since $P\sim_{\bU\bL} L_A$, there exist $F\in \bU_l(\bC \langle Z \rangle)$ and $G\in \bL_l(\bC \langle Z \rangle)$ such that 
    \begin{equation*}
        \bmat{P(X) & 0 \\ 0 & I} = F(X) \underbrace{\bmat{L_A(X) & 0 \\ 0 & I}}_{=: \ L_{\widetilde{A}}(X)} G(X) \foral X \in \bD_Q.
    \end{equation*}
    In particular, since every matrix in the above equation is invertible, we get
     \begin{equation}\label{eq:P.inverse.block.of.L.inverse}
        \bmat{P(X)^{-1} & 0 \\ 0 & I} = G(X)^{-1} L_{\widetilde{A}}(X)^{-1} F(X)^{-1} \foral X \in \bD_Q.
    \end{equation}
    Note that $F^{-1}\in \bU_l(\bC \langle Z \rangle)$ and $G^{-1}\in \bL_l(\bC \langle Z \rangle)$. Thus, \eqref{eq:P.inverse.block.of.L.inverse} shows that $P^{-1}$ is equal to the $(1,1)$ entry of $L_A^{-1}$, so that
    \begin{equation*}
        \|P^{-1}(X)\|\leq \big\| L_A^{-1}(X) \big\| \foral X\in \bD_Q.
    \end{equation*}
    Next, observe that
    \begin{equation*}
        \big\| L_A^{-1}(X) \big\| \leq \kappa(S) \big\| L_{B}^{-1}(X) \big\| \foral X\in \bD_Q.
    \end{equation*}
    Moreover, $\|B\|_{Q^\circ}\leq 1$ so that, for every $X\in\bD_Q$ and $r<1$, we can write $L_B(rX)^{-1}$ as a Neumann series and repeat the proof of Lemma \ref{lem:NGN.linear.pencil} to get $\|L_B(rX)^{-1}\|\leq \frac{1}{1-r}$ for all $X\in\bD_Q$. Putting all these inequalities together gives us
    \begin{equation}\label{eq:P.inverse.bound.conditional.number}
        \|\big(P^{(r)}\big)^{-1}\|_Q \leq \frac{\kappa(S)}{1-r} \foral r<1. 
    \end{equation}
    Recall the identity
    \begin{equation*}
        \big(P^{(r)}\big)^{-1}P=I+\big(P^{(r)}\big)^{-1}\big(P-P^{(r)}\big). 
    \end{equation*}
    We can then estimate the norm of the right hand side by applying \eqref{eq:P.inverse.bound.conditional.number} and Lemma \ref{lem:NGN.polynomials} to obtain the desired upper bound for $\sup_{r < 1} \big\| \big(P^{(r)} \big)^{-1} P\big\|_Q$.

    The second inequality can be obtained from an analogous argument.
\end{proof}

\begin{example}\label{eg:concrete.bounds}
    Let us compute some concrete bounds in Theorem \ref{mainthm:NGN.Q-stable.atom} and Proposition \ref{prop:improved.bounds}. 
    \begin{enumerate}
        \item Let $P(Z,W) = I - \frac{ZW}{2} - \frac{WZ}{2}$ be as in Example \ref{eg:Higman.linearization} and note that $P$ is $\fD_2$-stable. We observed that $P$ is an atom that exhibits a linearization
        \begin{equation*}
            \bmat{P(Z,W) & 0 \\ 0 & I_2} = \underbrace{\bmat{I & \frac{W}{\sqrt{2}} & \frac{Z}{\sqrt{2}} \\ 0 & I & 0 \\ 0 & 0 & I}}_{=: \ F} \underbrace{\bmat{I & \frac{-W}{\sqrt{2}} & \frac{-Z}{\sqrt{2}} \\ \frac{-Z}{\sqrt{2}} & I & 0 \\ \frac{-W}{\sqrt{2}} & 0 & I}}_{=: \ L_A} \underbrace{\bmat{I & 0 & 0 \\ \frac{Z}{\sqrt{2}} & I & 0 \\ \frac{W}{\sqrt{2}} & 0 & I}}_{=: \ G}.
        \end{equation*}
        Note that $A = (A_Z, A_W)$ satisfies
        \begin{equation*}
            \left\| A_Z \otimes X + A_W \otimes Y \right\| =\left\| \bmat{0 & \frac{Y}{\sqrt{2}} & \frac{X}{\sqrt{2}} \\ \frac{X}{\sqrt{2}} & 0 & 0 \\ \frac{Y}{\sqrt{2}} & 0 & 0} \right\| \leq 1 \foral (X, Y) \in \fD_2.
        \end{equation*}
        Then, \eqref{eqn:polar.dual.boundary.condition} shows that $A \in \overline{\fD_2^\circ}$, and we can therefore compute the bounds in \eqref{eqn:lem.NGN.linear.pencil.proof.3} and Proposition \ref{prop:improved.bounds} directly without having to find a similarity equivalence of $A$ with some $B \in \overline{\fD_2^\circ}$.
        
        Since both $F$, $G$ are of the form $I + J$ with $J^2 = 0$, we use the estimate
        \begin{equation*}
            \max\{\|I + J\|, \|I - J\|\} \leq 1 + \|J\|
        \end{equation*}
        to conclude that
        \begin{align*}
            \|F\|_{\fD_2} = \|F^{-1}\|_{\fD_2} &\leq 1 + \left\|\bmat{\frac{W}{\sqrt{2}} & \frac{Z}{\sqrt{2}}}\right\|_{\fD_2} = 2, \\
            \|G\|_{\fD_2} = \|G^{-1}\|_{\fD_2} &\leq 1 + \left\|\bmat{\frac{Z}{\sqrt{2}} & \frac{W}{\sqrt{2}}}\right\|_{\fD_2} = 2.
        \end{align*}
        Plugging these values into \eqref{eqn:lem.NGN.linear.pencil.proof.3}, we get
        the bound $\leq 52$ for both quantities. If we instead note that $\|P_2\|_{\fD_2} = 1$ and apply Proposition \ref{prop:improved.bounds} then we get the bound $\leq 3$.

        However, in this particular instance, there is a much simpler way of obtaining a better bound by noting that $P$ is of the form $I - \widetilde{P}$, where $\widetilde{P}(Z,W)$ is a strict contraction for each $Z, W \in \fD_2$. Thus, one can compute the Neumann series for $\big(P^{(r)}\big)^{-1} = \big( I - \widetilde{P}^{(r)} \big)^{-1}$ and proceed as in the proof of Lemma \ref{lem:NGN.linear.pencil} (for the case $F = I$) to show that
        \begin{equation*}
            \sup_{r < 1} \big\| \big(P^{(r)}\big)^{-1} P \big\|_{\fD_2} \leq 2 \qand \sup_{r < 1} \big\|P \big(P^{(r)}\big)^{-1} \big\|_{\fD_2} \leq 2.
        \end{equation*}
        In fact, the above inequalities are both equalities by noting that
        \begin{equation*}
            \big\| (P^{(0)})^{-1} P \big\|_{\fD_2} = \big\| P (P^{(0)})^{-1} \big\|_{\fD_2} = \|P\|_{\fD_2} = 2
        \end{equation*}
        The next example shows that this situation does not necessarily arise all the time, even for the case $k = 1$.

        \item Consider $P(Z,W) = I - \frac{2Z}{3} - \frac{2W}{3} + \frac{ZW}{3}$ over $\fD_2$ and obtain the linearization
        \begin{equation*}
            \bmat{P(Z,W) & 0 \\ 0 & I} = \bmat{I & -\frac{Z}{\sqrt{3}} \\ 0 & I} \underbrace{\bmat{I - \frac{2Z}{3} - \frac{2W}{3} & \frac{Z}{\sqrt{3}} \\ -\frac{W}{\sqrt{3}} & I}}_{=: \ L_A} \bmat{I & 0 \\ \frac{W}{\sqrt{3}} & I}.
        \end{equation*}
        That $L_A$ is irreducible is easy to check, so $P$ is an atom. Moreover, $P$ is $\fD_2$-stable as
        \begin{equation*}
            P(Z,W) = \frac{1}{3}\big((2I - Z)(2I - W) - I\big).
        \end{equation*}
        Indeed, if there exists a vector $v$ so that $P(Z,W)v = 0$, then we get that
        \begin{equation*}
            \|(2I-Z)(2I-W)v\| = \|v\|.
        \end{equation*}
        This cannot happen since
        \begin{equation*}
            \|(2I - W)v\| \geq 2\|v\| - \|Wv\| > \|v\|,
        \end{equation*}
        and similarly for $x = (2I - W)v$ we get
        \begin{equation*}
            \|(2I - Z)x\| > \|x\| > \|v\|.
        \end{equation*}
        Lastly, $\widetilde{P} = I - P$ is clearly not a contraction (e.g., take $Z = W = -0.9$). However, we note that $P$ has zeros on the boundary, e.g., $P(I,I) = 0$, so that the computation for the NGN-type bound is still a non-trivial exercise (otherwise $\|P^{-1}\|_Q < \infty$ and the NGN-type bound is simply $\|P\|_Q \|P^{-1}\|_Q$).

        Before applying the bounds from Proposition \ref{prop:improved.bounds}, we first need to compute a joint similarity $S$ as in Theorem \ref{thm:spec.rad.properties-SS}$(2)$ so that $S^{-1} A S \in \overline{\fD_2^\circ}$, since $A$ is irreducible but $A \not\in \overline{\fD_2^\circ}$ (check \eqref{eqn:polar.dual.boundary.condition} for $Z, W$ close to $I$). To this end, let $S = \bsmallmat{\sqrt{3} & 1 \\ 1 & \sqrt{3}}$ be the matrix whose columns are given by the eigenvectors of $A_Z + A_W$, and note that
        \begin{equation*}
            B := S^{-1} A S = \left( B_Z := \bmat{\frac{1}{2} & -\frac{1}{2 \sqrt{3}} \\ -\frac{1}{2 \sqrt{3}} & \frac{1}{6}}, B_W := \bmat{\frac{1}{2} & \frac{1}{2 \sqrt{3}} \\ \frac{1}{2 \sqrt{3}} & \frac{1}{6}} \right).
        \end{equation*}
        It follows that $B_Z = u u^t$ and $B_W = vv^t$ for the vectors
        \begin{equation*}
            u = \bmat{\frac{1}{\sqrt{2}} & -\frac{1}{\sqrt{6}}}^t \qand v = \bmat{\frac{1}{\sqrt{2}} & \frac{1}{\sqrt{6}}}^t,
        \end{equation*}
        and that, for any given $(X,Y) \in \fD_2$, we have
        \begin{equation*}
            \|B_Z \otimes X + B_W \otimes Y\| = \left\| \bmat{u \otimes I & v \otimes I} \bmat{ I \otimes X & 0 \\ 0 & I \otimes Y} \bmat{u^t \otimes I \\ v^t \otimes I} \right\| \leq \left\| \bmat{u & v} \right\|^2 = 1.
        \end{equation*}
        Thus, $B \in \overline{\fD_2^\circ}$ as desired. Lastly, note that $\kappa(S) = 2 + \sqrt{3}$, $\|P_1\|_{\fD_2} = \frac{4}{3}$ and $\|P_2\|_{\fD_2}=\frac{1}{3}$, so we can compute the bound in Proposition \ref{prop:improved.bounds} to obtain
        \begin{equation*}
            \sup_{r < 1} \big\| (P^{(r)})^{-1} P \big\|_{\fD_2} \leq 5 + 2 \sqrt{3} \approx 8.46.
        \end{equation*}
    \end{enumerate}
\end{example}

\section{Cyclicity of matrix free polynomials}\label{sec:cyclicity}

\subsection{Canonical weak-* topology}\label{subsec:canon.weak-*.top}

Shalit and the first named author established in \cite{Sampat-Shalit-Weak-star} that, for any given NC operator ball $\bD_Q$, $H^\infty(\bD_Q)$ exhibits a canonical weak-* topology by constructing a pre-dual as follows: for any $n\in \bN$ and $X\in \bD_Q(n)$, let $\Phi_X: F \mapsto F(X) \in M_n$ be the \emph{matrix evaluation map} on $H^\infty(\bD_Q)$. Next, for any $\eta\in M_n^*$, let $\varphi_{\eta,X}:=\eta\circ \Phi_X$. Clearly, $\varphi_{\eta,X}\in (H^\infty(\bD_Q))^*$, so that we can define
\begin{equation*}
    \cX(\bD_Q):=\overline{\spn}\{\varphi_{\eta,X}: X\in \bD_Q(n), \eta\in M_n^*, n\in\bN\}\subset (H^\infty(\bD_Q))^*.
\end{equation*}

By pre-dual of a Banach space $\cY$ we mean the following: let $\cX\subset \cY^*$ be a subspace of its dual. Using the canonical embedding of $\cY$ into its bidual $\cY^{**}$, we can view $\cY$ as a space of bounded linear functionals on $\cX$. If under this embedding we have $\cY=\cX^*$, then we say that $\cX$ is a pre-dual of $\cY$. In this case, we consider the $\sigma(\cY, \cX)$ topology on $\cY$ for which every $\varphi \in \cX$ is continuous. Note that in general a pre-dual need not exist, and if it exists, it need not be unique. 

\begin{theorem}[Theorem 3.1, \cite{Sampat-Shalit-Weak-star}]\label{thm:Sampat.Shalit.weak-*}
    $\cX(\bD_Q)$ is the unique pre-dual $\cX$ of $H^\infty(\bD_Q)$ such that the matrix point-evaluation maps $\Phi_X$ are $\sigma(H^\infty(\bD_Q),\cX)$ continuous for all $X\in \bD_Q$.
\end{theorem}

Note that $\cX(\bD_Q)$ was not shown to be an operator space pre-dual of $H^\infty(\bD_Q)$. We can nevertheless directly establish a similar result for each $M_k(H^\infty(\bD_Q)))=M_k\otimes H^\infty(\bD_Q)$ and $k \in \bN$. To this end, define
\begin{equation*}
    \cX_k(\bD_Q):=M_k^*\otimes \cX(\bD_Q),
\end{equation*}
and note that $\cX_k(\bD_Q)$ is a closed subspace of $(M_k\otimes H^\infty(\bD_Q))^*$. Let $E_{ij}$ be the standard matrix units in $M_k$ and let $\{e_{ij}\}$ be a dual basis of $M_k^*$. Then, every $\varphi \in \cX_k(\bD_Q)$ can be written uniquely as $\varphi=\sum_{i,j=1}^k e_{ij}\otimes \varphi_{ij}$ for appropriate $\varphi_{ij} \in \cX(\bD_Q)$. Lastly, for $X\in \bD_Q(n)$, let $\Phi_X: F \mapsto F(X) \in M_k(M_n)$ be the matrix point-evaluation map on $M_k(H^\infty(\bD_Q))$.

\begin{proposition}\label{prop:uniq.pre-dual.matrix.case}
    $\cX_k(\bD_Q)$ is the unique pre-dual $\cX_k$ of $M_k(H^\infty(\bD_Q))$ such that the matrix point-evaluation maps $\Phi_X$ are $\sigma(M_k(H^\infty(\bD_Q)),\cX_k)$ continuous for all $X\in \bD_Q$.
\end{proposition}

\begin{proof}
    We proceed exactly as in the proof of Theorem \ref{thm:Sampat.Shalit.weak-*}. To prove that $\cX_k(\bD_Q)$ is a pre-dual of $M_k(H^\infty(\bD_Q))$ there are two things we need to establish (see \cite[Section 3]{Sampat-Shalit-Weak-star} and \cite[Section 2]{Davidson-Wright-predual} for details):
    \begin{enumerate}
        \item $\cX_k(\bD_Q)$ norms $M_k(H^\infty(\bD_Q))$, i.e.,
        \begin{equation}\label{eqn:pre-dual.norms.dual}
            \sup \{\|\varphi(F)\| \ : \ \|\varphi\| \leq 1, \ \varphi \in \cX_k(\bD_Q)\} = \|F\|_Q \foral F \in M_k(H^\infty(\bD_Q)).
        \end{equation}
        \item The closed unit ball of $M_k(H^\infty(\bD_Q))$ is $\sigma(M_k(H^\infty(\bD_Q)),\cX_k(\bD_Q))$-compact.
    \end{enumerate}

    \underline{Proof of (1)}: The inequality $\leq$ in \eqref{eqn:pre-dual.norms.dual} is trivial, so we need only prove the inequality $\geq$. For any $X\in \bD_Q(n)$ and $\eta\in (M_k(M_n))^*=M_k^*\otimes M_n^*$, define $\varphi_{\eta,X}:=\eta\circ \Phi_X$. We claim that any such $\varphi_{\eta,X}$ is an element of $\cX_k(\bD_Q)$. Indeed, we can write $\eta=\sum_j \eta_j^{(k)}\otimes \eta_j^{(n)}$ for finitely many $\eta_j^{(k)} \in M_k^*$ and $\eta_j^{(n)} \in M_n^*$. One then checks that
    \begin{equation*}
        \varphi_{\eta,X}=\sum_j \eta_j^{(k)}\otimes \varphi_{\eta_j^{(n)},X}\in \cX_k(\bD_Q).
    \end{equation*}
    Thus, we get $\varphi_{\eta, X} \in \cX_k(\bD_Q)$, as claimed. Now, given any $F \in M_k(H^\infty(\bD_Q))$ and $X \in \bD_Q(n)$, let $\eta_X \in (M_k(M_n))^*$ be such that $\|\eta_X\| = 1$ and
    \begin{equation*}
        |\varphi_{\eta_X, X}(F)| = |\eta_X(F(X))| = \|F(X)\|. 
    \end{equation*}
    As $F$ was chosen arbitrarily, and since $\|\varphi_{\eta_X, X}\| \leq \|\eta_X\|\|\Phi_X\| \leq 1$, we take the supremum as $X \in \bD_Q$ above to conclude that the inequality $\geq$ holds in \eqref{eqn:pre-dual.norms.dual}, as required.

    \underline{Proof of (2)}: Let $F_\iota=\sum_{i,j=1}^k E_{ij}\otimes f_\iota^{(ij)}$ be a net in the closed unit ball of $M_k(H^\infty(\bD_Q))$. Then, for each $1\leq i,j\leq k$, $(f_\iota^{(ij)})$ is a net in the closed unit ball of $H^\infty(\bD_Q)$. Hence, using the fact that the closed unit ball of $H^\infty(\bD_Q)$ is $\sigma(H^\infty(\bD_Q), \cX(\bD_Q))$ compact (by Theorem \ref{thm:Sampat.Shalit.weak-*} and the Banach-Alaoglu Theorem), we can find a sub-net $(F_{\iota_\lambda})$ such that for every $1\leq i,j\leq n$, $f_{\iota_\lambda}^{(ij)} \xrightarrow{w^*} f^{(ij)}$ for some $f^{(ij)}$ in the closed unit ball of $H^\infty(\bD_Q)$. It is then straightforward to check, using the definition of $\cX_k(\bD_Q)$, that
    \begin{equation*}
        F_{\iota_\lambda}=\sum_{i,j=1}^k E_{ij}\otimes f_{\iota_\lambda}^{(ij)} \to \sum_{i,j=1}^k E_{ij}\otimes f^{(ij)},
    \end{equation*}
    in $\sigma(M_k(H^\infty(\bD_Q)),\cX_k(\bD_Q))$, and the limit lies in the closed unit ball of $M_k(H^\infty(\bD_Q))$ (using (1)). This completes the proof of (2), and so $\cX_k(\bD_Q)$ is a pre-dual of $M_k(H^\infty(\bD_Q))$.

    Finally, the proof of continuity of the matrix point-evaluations and of the uniqueness of the pre-dual works exactly as in the proof of Theorem \ref{thm:Sampat.Shalit.weak-*}, so we will not repeat it here. 
\end{proof}

\begin{proposition}\label{prop:weak-*.topology.properties}
The weak-* topology on $M_k(H^\infty(\bD_Q))$ has the following properties.
    \begin{enumerate}
        \item A net $(F_\iota) \subset M_k(H^\infty(\bD_Q))$ converges weak-* to some $F\in M_k(H^\infty(\bD_Q))$ if and only if every entry of $(F_\iota)$ converges weak-* in $H^\infty(\bD_Q)$ to the corresponding entry of $F$;
        \item The weak-* topology on $M_k(H^\infty(\bD_Q))$ coincides with the topology of pointwise convergence on bounded sets;
        \item $M_k(\bC \langle Z \rangle)$ is weak-* dense in $M_k(H^\infty(\bD_Q))$;
        \item  Fix $F\in M_k(H^\infty(\bD_Q))$. The linear maps of left multiplication by $F$, $G\mapsto FG$, and of right multiplication by $F$, $G\mapsto GF$, are both weak-* continuous on $M_k(H^\infty(\bD_Q))$.
    \end{enumerate}
\end{proposition}

\begin{proof}
    \begin{enumerate}
        \item This follows from a straightforward argument and the definition of $\cX_k(\bD_Q)$.
        \item As observed in \cite{Sampat-Shalit-Weak-star}, a bounded net $(f_\iota)$ in $H^\infty(\bD_Q)$ converges weak-* to $f\in H^\infty(\bD_Q)$ if and only if $(f_\iota)$ converges to $f$ pointwise on $\bD_Q$. (1) then implies that the exact same statement holds true for $M_k(H^\infty(\bD_Q))$ and its weak-* topology.
        \item For $k=1$, this follows from \cite[Proposition 3.4 (2)]{Sampat-Shalit-Weak-star}. A straightforward application of (1) then shows that the statement holds for all $k\in\bN$.
        \item By (1), it suffices to prove the claim for $k=1$. To this end, for every $\varphi \in \cX(\bD_Q)$, we define the bounded linear functional $\varphi^F$ by 
        \begin{equation*}
        \varphi^F:H^\infty(\bD_Q) \to \C,\quad \varphi^F(G):=\varphi(FG).
        \end{equation*}
        We will prove that $\varphi^F\in \cX(\bD_Q)$ for every $\varphi \in \cX(\bD_Q)$, from which it follows that left multiplication by $F$ is weak-* continuous.

        First, consider $\varphi_{\eta,X} \in \cX(\bD_Q)$ for some $X\in \bD_Q(n)$ and $\eta\in M_n^*$. Then, for every $G\in H^\infty(\bD_Q)$ we have
        \begin{equation}\label{eq:varphi.power.F}
        \varphi_{\eta,X}^F(G)=\varphi_{\eta,X}(FG)=\eta(F(X)G(X))=\varphi_{\eta(F(X)\cdot),X}(G).
        \end{equation}
        Observe that $\varphi_{\eta(F(X)\cdot),X}\in \cX(\bD_Q)$. Indeed, $A\mapsto \eta(F(X)A)$ clearly defines an element of $M_n^*$. Therefore, $\varphi_{\eta,X}^F \in \cX(\bD_Q)$ holds using \eqref{eq:varphi.power.F}. From this it immediately follows that $\varphi^F\in \cX(\bD_Q)$ for all
        \begin{equation*}
            \varphi\in \mathrm{span}\{\varphi_{\eta,X} \ : \ X\in \bD_Q(n), \ \eta\in M_n^*, \ n\in\bN\}.
        \end{equation*}
        Next, let $\varphi\in \cX(\bD_Q)$ be arbitrary. Then there exists a sequence $(\varphi_n)$ in the linear span above such that $\varphi_n\to  \varphi$ in $(H^\infty(\bD_Q))^*$. A standard argument shows that $\varphi_n^F\to \varphi^F$ in $(H^\infty(\bD_Q))^*$. This implies $\varphi^F\in \cX(\bD_Q)$, as desired. 
    
        The statement about right multiplication by $F$ follows similarly. 
    \end{enumerate}
\end{proof}

\subsection{Weak-* cyclicity}\label{subsec:cyclicity.matrix.poly}

We define the left ideal $\fI_L(F)$ and the right ideal $\fI_R(F)$ of any given $F \in M_k(H^\infty(\bD_Q))$ as
\begin{align*}
    \fI_L(F) &:= \{ G F \ : \ G \in M_k(H^\infty(\bD_Q)) \}, \\
    \fI_R(F) &:= \{ F G \ : \ G \in M_k(H^\infty(\bD_Q)) \}.
\end{align*}
We then use the notation
\begin{equation*}
    [F]_L := \overline{\fI_L(F)}^{w^*} \qand [F]_R := \overline{\fI_R(F)}^{w^*}
\end{equation*}
respectively for the \emph{left} and \emph{right invariant subspaces} generated by $F$.

\begin{definition}\label{def:cyclicity}
    An NC function $F \in M_k(H^\infty(\bD_Q))$ is said to be \emph{left} (resp. \emph{right}) \emph{cyclic} if $[F]_L = M_k(H^\infty(\bD_Q))$ (resp. $[F]_R = M_k(H^\infty(\bD_Q))$).
\end{definition}

\begin{lemma}\label{lem:cyclic.iff.identity}
    $F \in M_k(H^\infty(\bD_Q))$ is left (resp. right) cyclic if and only if $I\in [F]_L$ (resp. $I\in [F]_R$).
\end{lemma}

\begin{proof}
    If $F$ is left cyclic, then $I\in [F]_L$ by definition.
    
    Conversely, if $I\in [F]_L$, then there is a net $(G_\iota)\subset M_k(H^\infty(\bD_Q))$ such that $G_\iota F \xrightarrow{w^*} I$. Since left multiplication is weak-* continuous by Proposition \ref{prop:weak-*.topology.properties}(4), it follows that $M_k(H^\infty(\bD_Q))=[F]_L$.

    The proof for right cyclicity is similar.
\end{proof}

\begin{remark}
    As noted in \cite{Pascoe-example}, there exists an NC function $F$ such that
    \begin{enumerate}
        \item $F$ is analytic on $s \fB_d$ in the uniform topology for all $s > 1$, and
        \item $F$ is unbounded on $\fB_d$.
    \end{enumerate}
    Therefore, the conclusion of the following lemma is non-trivial in the NC setup, whereas the commutative counterpart is rather easy to check.
\end{remark}

\begin{lemma}\label{lem:boundedness.of.P^{(r)}}
    Let $P\in M_k(\bC \langle Z \rangle)$ be $Q$-stable. Then $(P^{(r)})^{-1}\in M_k(H^\infty(\bD_Q))$ for all $r<1$.
\end{lemma}

\begin{proof}
    We can without loss of generality assume that $P(0)=I$. By Lemma \ref{lem:Higman's.linearization.trick}, $P\sim L_A$ for some linear pencil $L_A$. Thus there exist $F,G\in GL_l(\bC \langle Z \rangle)$ such that
    \begin{equation*}
   \bmat{P(X) & 0 \\ 0 & I} = F(X) \underbrace{\bmat{L_A(X) & 0 \\ 0 & I}}_{=: \ L_{B}(X)}G(X) \foral X \in \bD_Q.
\end{equation*}
Note that since $P$ is $Q$-stable, all matrices in the above equation are invertible. Hence, dilating by $r<1$ and taking the inverse on both sides of the equation yields
\begin{equation*}
    \big\|\big(P^{(r)}\big)^{-1}\big\|_Q \leq \big\|\big(G^{(r)}\big)^{-1}\big\|_Q \big\|\big(L_{B}^{(r)}\big)^{-1}\big\|_Q\big\|\big(F^{(r)}\big)^{-1}\big\|_Q.
\end{equation*}
Since, $\frac{1}{r}\bD_Q\subset \dom (L_{B}^{(r)})^{-1}$, we get $\big\|\big(L_{B}^{(r)}\big)^{-1}\big\|_Q<\infty$ from \cite[Corollary 3.2]{Shalit-Shamovich-spec-rad}. Using the fact that $(G^{(r)})^{-1}, (F^{(r)})^{-1} \in M_l(\bC \langle Z \rangle)$ we therefore conclude that $(P^{(r)})^{-1} \in M_k(H^\infty(\bD_Q))$.
\end{proof}

We are now sufficiently prepared to prove Theorem \ref{mainthm:cyc.matrix.free.poly}.

\subsection*{\texorpdfstring{Proof of Theorem \ref{mainthm:cyc.matrix.free.poly}}{Proof of Theorem B}} As before, we only prove the statement for left cyclicity. 

Suppose $P \in M_k(\bC \langle Z \rangle)$ is left cyclic. By definition, there exists a net $(G_\iota) \subset M_k(H^\infty(\bD_Q))$ such that $G_\iota P \xrightarrow{w^*} I$. In particular, the weak-* continuity of the matrix point-evaluation map $\Phi_X$ implies $G_\iota(X) P(X) \to I$ for all $X \in \bD_Q$. The continuity of the determinant then shows $\det G_\iota(X) \det P(X) \to 1$ for all $X\in \bD_Q$, so that $P$ is $Q$-stable.

Conversely, suppose $P \in M_k(\bC \langle Z \rangle)$ is $Q$-stable and let $P = P_1P _2 \dots P_l$ be its atomic factorization. Clearly, each $P_j$ is also $Q$-stable. Let $(r_n) \subset (0,1)$ be such that $r_n \to 1^-$ and observe how for every $n$, $P_1^{(r_n)}$ is $Q$-stable and $(P_1^{(r_n)})^{-1} \in M_k(H^\infty(\bD_Q))$ by Lemma \ref{lem:boundedness.of.P^{(r)}}. Next, let $G\in M_k(H^\infty(\bD_Q))$ be arbitrary and consider the sequence $(H_n)$ in $M_k(H^\infty(\bD_Q))$ given by 
\begin{equation*}
    H_n:=G(P_1^{(r_n)})^{-1}P_1P_2\dots P_l. 
\end{equation*}
By Theorem \ref{mainthm:NGN.Q-stable.atom}, $(H_n)$ is a bounded sequence. Moreover, since $((P_1^{(r_n)})^{-1}P_1)$ converges pointwise to $I$, $(H_n)$ converges pointwise to $GP_2\dots P_l$. Therefore, $H_n \xrightarrow{w^*} GP_2\dots P_l$ by Proposition \ref{prop:weak-*.topology.properties}(2). As $G$ was chosen arbitrarily, we get
\begin{equation*}
    \fI_L(P_2\dots P_l)\subset [P_1P_2\dots P_l]_L,
\end{equation*}
which in turn implies 
\begin{equation*}
    [P_2\dots P_l]_L\subset [P_1P_2\dots P_l]_L.
\end{equation*}
Continuing this line of reasoning, we obtain 
\begin{equation*}
    [I]_L\subset [P_l]_L\subset [P_{l-1}P_l]_L\subset \dots \subset [P_2\dots P_l]_L\subset [P_1P_2\dots P_l]_L.
\end{equation*}
Hence, $I\in [P]_L$, so that $P$ is left cyclic by Lemma \ref{lem:cyclic.iff.identity}. \hfill \qedsymbol

\subsection{Cyclic polynomials in the free Hardy space}\label{subsec:cyclic.poly.free.Hardy}

As mentioned in the introduction, Theorem \ref{mainthm:cyc.matrix.free.poly} allows us to recover a known result about the cyclicity of free polynomials in the free Hardy space (see \cite{AroraPhD, Arora-Augat-Jury-Sargent-free-OPA, JMS-ratFock}). In the same vein, in Section \ref{sec:cyclicity.rational} we recover a known result about cyclicity of NC rational functions in the free Hardy space. 

Recall that the free Hardy space $\bH_d^2$ is defined to be the Hilbert space of all power series in $d$ freely non-commuting variables with square-summable (scalar) coefficients: 
\begin{equation*}
    \bH_d^2=\left\{f(Z)=\sum_{w\in \F}a_w Z^w: \sum_{w\in \F}|a_w|^2<\infty\right\},
\end{equation*}
under the inner-product induced by the $\ell^2$ inner-product of the coefficients. One can show that for every $f\in \bH_d^2$ and $X$ in the $d$-dimensional unit row ball $\fB_d$, $f(X)$ converges (see \cite[Theorem 1.1]{Pop-NC-appl}), so that every element of $\bH_d^2$ defines an NC function on $\fB_d$. Equipping $M_k$ with the Hilbert-Schmidt inner product allows us to consider the Hilbert space $M_k(\bH_d^2)=M_k\otimes \bH_d^2$, which we interpret as a matrix-valued version of the free Hardy space. Every element $F$ of $M_k(\bH_d^2)$ can we written as $F(Z)=\sum_w A_w Z^w$ for a family of matrices $A_w$ in $M_k$. If $G(Z)=\sum_w B_w Z^w$ is another element of $M_k(\bH_d^2)$, then the inner product of $F$ with $G$ is given by $\sum_w \Tr({B_w}^* A_w)$. For any given $F \in M_k(\bH^2_d)$ and $X\in\fB_d$, $F(X)$ is to be interpreted as $F(X)=\sum_w A_w\otimes X^w$. In this way, every element of $M_k(\bH_d^2)$ defines a matrix-valued NC function on $\fB_d$.

One can show that $\bH_d^2$ is an NC reproducing kernel Hilbert space (NCRKHS) in the sense of Ball, Marx and Vinnikov \cite{BMV-NCRKHS}, meaning that for every point $X\in \fB_d$, the linear map $f \mapsto f(X)$ is continuous from $\bH_d^2$ to the space of square matrices of appropriate size. Clearly, for every $X\in \fB_d$, this implies that $F \mapsto F(X)$ also defines a continuous linear map from $M_k(\bH_d^2)$ to the space of square matrices of appropriate size. To every NCRKHS one can naturally associate an algebra, called its \emph{(left) multiplier algebra}, consisting of all NC functions on the underlying domain that multiply (on the left) the NCRKHS into itself. One can show that the left multiplier algebra of $\bH_d^2$ is completely isometrically isomorphic to $H^\infty(\fB_d)$; see \cite[Theorem 3.1]{SSS-algebras} and \cite[Theorem 3.1]{Pop-NC-appl}. Therefore, $H^\infty(\fB_d)\subset \bH_d^2$, and there is a natural weak operator topology (WOT) on $H^\infty(\fB_d)$. In fact, Davidson and Pitts prove that the weak operator topology and the weak-* topology on $H^\infty(\fB_d)$ coincide; see \cite[Corollary 2.12]{Davidson-Pitts-inv}.

The above described identification of $H^\infty(\fB_d)$ with the left multiplier algebra of $\bH_d^2$ is what makes it possible to use our results to say something about cyclicity in $\bH_d^2$. To this end, we define for any given $F\in M_k(\bH_d^2)$,
\begin{align*}
    \fI^2(F) := \{ G F \ : \ G \in M_k(H^\infty(\fB_d)) \} \qand  [F]^2 := \overline{\fI^2(F)}^{\|\cdot\|}.
\end{align*}
We then say that $F$ is \emph{(left) cyclic} in $M_k(\bH_d^2)$ if $[F]^2=M_k(\bH_d^2)$. In this setting it is straightforward to show that cyclicity of $F$ is equivalent to $I\in [F]^2$. 

\begin{corollary}\label{cor:cyclicity.free.Hardy.space}
    $P \in M_k(\bC \langle Z \rangle)$ is cyclic in $M_k(\bH_d^2)$ if and only if $P$ is $\fB_d$-stable.
\end{corollary}

\begin{proof}

     If $P$ is left cyclic, then the proof of $P$ being $\fB_d$-stable proceeds exactly as the first part of the proof of Theorem \ref{mainthm:cyc.matrix.free.poly}, using the fact that point evaluations define bounded linear functionals on $M_k(\bH_d^2)$.

     Conversely, assume that $P$ is $\fB_d$-stable. Since $P\in M_k(H^\infty(\fB_d))$, Theorem \ref{mainthm:cyc.matrix.free.poly} shows that $P$ is left cyclic in $M_k(H^\infty(\fB_d))$. By definition we can find a net $(H_\iota) \subset M_k(H^\infty(\fB_d))$ such that $H_\iota P \xrightarrow{w^*} I$.
     Applying Proposition \ref{prop:weak-*.topology.properties}(1) and the fact that the weak operator topology coincides with the weak-* topology on $H^\infty(\fB_d)$, it is straightforward to check that each entry of $(H_\iota P)$ converges weakly to the corresponding entry of the constant NC function $I$ in $\bH^2_d$. It then follows from a standard argument that $H_\iota P \xrightarrow{w} I$ in $M_k(\bH^2_d)$.
     %In particular, every entry of the net $H_\iota P$ converges weak-* in $H^\infty(\fB_d)$ to the corresponding entry of $I$ (see Proposition \ref{prop:weak-*.topology.properties}(1)). As noted before the lemma statement, this in turn implies weak convergence in $\bH_d^2$ of every entry. Next, from a straightforward calculation we obtain weak convergence of $H_\iota P$ to $I$ in $M_k(\bH_d^2)$.
     Therefore, it follows that $I\in \overline{\fI^2(P)}^{\text{w}}$. However, because $\fI^2(P)$ is convex, its weak closure coincides with its norm closure in $M_k(\bH_d^2)$, and so $I\in [P]^2$, as required.
\end{proof}

\section{Cyclicity of NC rational functions}\label{sec:cyclicity.rational}

Following \cite{HMV-NC.rat} and \cite{Volcic-rationals}, an \emph{NC rational expression} in $d$ freely non-commuting variables $Z_1,\dots,Z_d$, is any syntactically valid expression we can build out of NC polynomials in $d$ variables by addition, multiplication and taking inverses. For example, 
\begin{equation*}
    (Z_1Z_2-Z_2Z_1)^{-1}Z_3 \qand Z_1Z_2^{-1}-3(Z_1Z_3^{-1}+Z_2)^{-1},
\end{equation*}
are NC rational expressions. We say that an NC rational expression is \emph{non-degenerate} if there is a tuple $X\in\bM^d$ at which it can be evaluated.
%This means that after plugging $X$ into to NC rational expression, all involved inverses exist.
The domain of a non-degenerate NC rational expressions in $d$ variables is the set of all points in the $d$-dimensional NC universe $\bM^d$ at which the NC rational expression can be evaluated. Next, two non-degenerate NC rational expressions are said to be equivalent if they evaluate to the same matrix at every point in the intersection of their domains. It is readily checked that this defines an equivalence relation. Finally, an \emph{NC rational function} is an equivalence class of non-degenerate NC rational expression, and its domain is the union of the domains of all NC rational expressions in the equivalence class. The NC rational functions form a skew field, denoted by $\fskew$, which is the universal skew field of fractions of the ring of NC polynomials $\bC\langle Z\rangle$; see \cite{Cohn-skew-fields} for background on universal skew fields. 

Before we get to the proof of Theorem \ref{mainthm:cyc.NC.rationals}, note that whenever $\fr$ is an NC rational function such that $s\bD_Q\subset \dom \fr$ for some $s>1$, then $\fr$ will be bounded on every ball $r\bD_Q$ with $r<s$; see \cite[Corollary 3.2]{Shalit-Shamovich-spec-rad}. In particular, in this case we have $\fr\in H^\infty(\bD_Q)$. It should be noted that this is not true for general NC functions, even when $\bD_Q=\fB_d$; see \cite{Pascoe-example}. Our proof of Theorem \ref{mainthm:cyc.NC.rationals} follows the same path as in \cite[Theorem 3.7]{AroraPhD}, which relies on a certain factorization formula (see \eqref{eq:factorization.Arora.PhD} below) for matrices of NC rational functions.

\subsection*{\texorpdfstring{Proof of Theorem \ref{mainthm:cyc.NC.rationals}}{Proof of Theorem C}} %As before, we only prove the statement for left cyclicity.
The $Q$-stability of any given left/right cyclic $\fr$ as in the hypothesis follows exactly as in Theorem \ref{mainthm:cyc.matrix.free.poly}.

To prove sufficiency, assume $\fr = (\fr_{\lambda,\mu})$ is $Q$-stable. For all $1\leq \lambda,\mu\leq d$ there exists a triple $(A_{\lambda,\mu},b_{\lambda,\mu},c_{\lambda,\mu})$, called a \emph{descriptor realization}, where $A_{\lambda,\mu}=(A^{\lambda,\mu}_1,\dots,A^{\lambda,\mu}_d)\in M_{n_{\lambda,\mu}}^d$ and $b_{\lambda,\mu},c_{\lambda,\mu}\in \bC^{n_{\lambda,\mu}}$, such that   
\begin{equation*}
    \dom \fr_{\lambda,\mu}\subset \{X\in \bM^d: L_{A_{\lambda,\mu}}(X)\text{ is invertible}\},
\end{equation*}
and 
\begin{equation*}
    \fr_{\lambda,\mu}(Z)=b_{\lambda,\mu}^* L_{A_{\lambda,\mu}}(Z)^{-1} c_{\lambda,\mu},
\end{equation*}
as functions on $\dom \fr_{\lambda,\mu}$. Here, for $X \in \dom \fr_{\lambda,\mu}$ the above equation is to be interpreted as
\begin{equation*}
    \fr_{\lambda,\mu}(X)=b_{\lambda,\mu}^*\otimes I L_{A_{\lambda,\mu}}(X)^{-1} c_{\lambda,\mu}\otimes I.
\end{equation*}
For a proof of this fact, see for example \cite{Volcic-rationals}. Next, we construct a descriptor realization for $\fr$. For every $1\leq j\leq d$, define
\begin{align*}
    A_j:=\diag(A^{1,1}_j,A^{1,2}_j\dots,A^{1,k}_j,A^{2,1}   _j,A^{2,2}_j,\dots  ,A^{k,1}_j,\dots,A^{k,k}_j).
\end{align*}
Also, let 
\begin{align*}
    b:=\bmat{b_{1,1}& 0 & \dots & 0\\ \vdots & \vdots & \ddots& \vdots \\ b_{1,k}& 0 & \dots & 0\\ 0 & b_{2,1} & \dots & 0\\
    \vdots & \vdots & \ddots&\vdots \\ 0 & b_{2,k} &\dots & 0 \\ \vdots & \vdots &\vdots&\vdots \\ \vdots & \vdots & \vdots & \vdots\\ 0 & 0 &\dots & b_{k,1}\\ \vdots& \vdots &\ddots &\vdots\\ 0 & 0 & \dots & b_{k.k}}, \qquad
    c:=\bmat{c_{1,1}&\dots& 0\\ \vdots& \ddots& \vdots\\ 0&\dots& c_{1,k}\\c_{2,1}&\dots& 0\\ \vdots& \ddots& \vdots\\ 0&\dots& c_{2,k} \\ \vdots & \vdots & \vdots\\ \vdots & \vdots & \vdots\\
    c_{k,1}&\dots& 0\\ \vdots& \ddots& \vdots\\ 0&\dots& c_{k,k}}.
\end{align*}
It is readily checked that
\begin{equation}\label{eq:domain.of.big.pencil}
    \dom \fr= \bigcap_{1\leq \lambda,\mu\leq k}\dom \fr_{\lambda,\mu} \subset \{X\in \bM^d: L_A(X)\text{ is invertible}\},
\end{equation}
and 
\begin{equation}\label{eq:realization.for.r}
    \fr(X)=b^*\otimes I L_A(X)^{-1} c\otimes I \foral X \in \dom \fr.
\end{equation}
Since $s\bD_Q\subset \dom \fr$ by assumption, we get from \eqref{eq:domain.of.big.pencil} that $s\bD_Q$ is contained in the invertibility domain of $L_A(Z)$, so that $\rho_{Q^\circ}(A)\leq \frac{1}{s}<1$ by Theorem \ref{thm:spec.rad.dom.pencil-SS}. Therefore, by Theorem \ref{thm:spec.rad.properties-SS}, $A$ is jointly similar to some $B\in \bD_Q^\circ$ via an invertible matrix $S$ of appropriate size. Hence, for all $X\in \bD_Q$ we get $\|B\otimes X\|\leq \|B\|_{Q^\circ}\|X\|_Q<1$, where $B\otimes X=\sum_j B_j\otimes X_j$, so that we can expand $L_B(X)^{-1}$ as a Neumann series, which yields
\begin{equation*}
    L_A(X)=S\otimes I L_B(X)^{-1} S^{-1}\otimes I=S\otimes I \left(\sum_{n=0}^\infty (B\otimes X)^n\right) S^{-1}\otimes I,
\end{equation*}
and
\begin{equation}\label{eq:inverse.pencil.bounded}
    \|L_A(X)^{-1}\|\leq \kappa(S) \left\|\sum_{n=0}^\infty (B\otimes X)^n\right\|\leq \frac{\kappa(S)}{1-\|B\|_{Q^\circ}}.
\end{equation}
In particular, $L_A(Z)^{-1}\in M_N(H^\infty(\bD_Q))$, where $N=\sum_{1\leq \lambda,\mu\leq d} n_{\lambda,\mu}$.
    
Using (\ref{eq:realization.for.r}), one can check that
\begin{equation}\label{eq:factorization.Arora.PhD}
    \underbrace{\bmat{I & 0 & I\\ c& L_A(Z) & 0\\ 0 & b^* & 0}}_{=: \ \alpha(Z)}=\underbrace{\bmat{I&0&0\\ c& L_A(Z)&0\\ 0&b^*&I}}_{=: \ \beta(Z)}\bmat{I&0&0\\0&I&0\\0&0&\fr(Z)}\underbrace{\bmat{I&0&I\\0&I&-L_A(Z)^{-1}c\\0&0&I}}_{=: \ \gamma(Z)},
\end{equation}
as functions on $s\bD_Q$. Clearly, $\alpha(Z), \beta(Z)\in M_e(\bC \langle Z \rangle)$ and $\gamma(Z)\in  M_e(H^\infty(\bD_Q))$ by \eqref{eq:inverse.pencil.bounded}, where $e=N+2k$. Since $\gamma(Z)$ is an upper uni-triangular matrix in $M_e(H^\infty(\bD_Q))$, it is straightforward to calculate its inverse and check that $\gamma(Z)^{-1}\in M_e(H^\infty(\bD_Q))$, again by \eqref{eq:inverse.pencil.bounded}. We can now turn \eqref{eq:factorization.Arora.PhD} into
\begin{equation}\label{eq:factorization.Arora.PhD.rearranged}
    \alpha(Z)\gamma(Z)^{-1}=\beta(Z)\bmat{I&0&0\\0&I&0\\0&0&\fr(Z)},
\end{equation}
as functions on $s\bD_Q$, which yields 
\begin{equation*}
    \det\alpha(X)\det\gamma(X)^{-1}=\det\beta(X)\det\fr(X)\foral X\in\bD_Q.
\end{equation*}
Since $\fr$ is assumed to be $Q$-stable, this shows that $\alpha$ is $Q$-stable as well. Thus, $\alpha$ is left cyclic by Theorem \ref{mainthm:cyc.matrix.free.poly}. Therefore, we can find a net $(H_\iota)$ in $M_e(H^\infty(\bD_Q))$ such that $H_\iota \alpha \xrightarrow{w^*} \gamma$. Next, multiplying both sides of \eqref{eq:factorization.Arora.PhD.rearranged} by $H_\iota$ yields
\begin{equation}\label{eq:factorization.Arora.PhD.multipliedH}
    H_\iota(Z)\alpha(Z)\gamma(Z)^{-1}=H_\iota(Z)\beta(Z)\bmat{I&0&0\\0&I&0\\0&0&\fr(Z)}=\bmat{* & * & * \\ * & * & * \\ * & * & J_\iota(Z)\fr(Z)},
\end{equation}
where $J_\iota(Z)$ is an element of $M_k(H^\infty(\bD_Q))$. By Proposition \ref{prop:weak-*.topology.properties}(4), the left hand side of \eqref{eq:factorization.Arora.PhD.multipliedH} converges weak-* to $I$. Since weak-* convergence is equivalent to entrywise weak-* convergence by Proposition \ref{prop:weak-*.topology.properties}(1), from \eqref{eq:factorization.Arora.PhD.multipliedH} we obtain $J_\iota\fr \xrightarrow{w^*} I$. Hence, $\fr$ is left cyclic by Lemma \ref{lem:cyclic.iff.identity}.

For the proof of right cyclicity, we proceed from \eqref{eq:factorization.Arora.PhD} similarly to \eqref{eq:factorization.Arora.PhD.rearranged} except with $\beta(Z)^{-1}$ on the LHS and $\gamma(Z)$ in the RHS. The rest of the argument is exactly the same, and so this completes the proof. \hfill \qedsymbol

\vspace{2mm}

Since the domain of every NC rational function in $H^\infty(\fB_d)$ contains a row ball of radius strictly greater than one (see \cite[Theorem A]{JMS-ratFock}), the following corollary is immediate. 

\begin{corollary}\label{cor:cyclicity.rational.Hardy.space}
    A matrix of NC rational functions $\fr\in M_k(H^\infty(\fB_d))$ is left/right cyclic if and only if it is $\fB_d$-stable.
\end{corollary}

We also obtain the following result which was first observed in \cite[Theorem C]{JMS-ratFock} for $k=1$; see also \cite[Theorem 3.7]{AroraPhD} for an alternate proof.

\begin{corollary}\label{cor:cyclic.NC.rat.Hardy}
    A matrix of NC rational functions $\fr\in M_k(\bH_d^2)$ is cyclic if and only if it is $\fB_d$-stable.
\end{corollary}

\begin{proof}
    First, note that an NC rational function is an element of $\bH_d^2$ if and only it is in $H^\infty(\fB_d)$; see \cite[Theorem A]{JMS-ratFock}. Thus, since $\fr\in M_k(\bH_d^2)$ by assumption, we have $\fr\in M_k(H^\infty(\fB_d))$. The rest of the proof is identical to the proof of Corollary \ref{cor:cyclicity.free.Hardy.space} (using Corollary \ref{cor:cyclicity.rational.Hardy.space}).
\end{proof}

\subsection{NC parallel sum function}\label{subsec:NC.parallel.sum} Up to this point, we have used approximation techniques involving the domain of a stable NC rational function to determine its cyclicity. However, we present a class of NC functions in this subsection that demonstrates cyclicity by analyzing the range of these functions. In particular, we introduce the \emph{NC parallel sum function} on the NC bidisk $\fD_2$.

Given a finite-dimensional Hilbert space $\cH$ and any two Hermitian positive semi-definite (p.s.d.) operators $A, B \in \sB(\cH)$ (i.e., $A, B \succeq 0$ in the L\"owner order), the \emph{parallel sum} $A : B$ is defined as
\begin{equation*}
    A : B = A (A + B)^\dagger B.
\end{equation*}
Here, $(A + B)^\dagger$ represents the Moore--Penrose (generalized) inverse of $A + B$. When $A$ and $B$ are also non-singular, we have the following equivalent ways of defining the parallel sum:
\begin{equation*}
    A : B = A (A+B)^{-1} B = (A^{-1} + B^{-1})^{-1} = B (A + B)^{-1} A = B : A.
\end{equation*}

The parallel sum of matrices was introduced by Anderson and Duffin \cite{AD-parallel-sum} to generalize a certain network synthesis procedure, and was later generalized to bounded, non-invertible operators on infinite-dimensional spaces by Fillmore and Williams \cite{FW-parallel-sum-inft-dim} through the study of operator ranges. Since then, the parallel sum has appeared in several influential papers in various contexts: (i) the foundation for harmonic operator means by Kubo and Ando \cite{Kubo-Ando-harmonic-mean}, (ii) a streamlined proof of Lieb and Ruskai's theorem on strong sub-additivity of quantum mechanical entropy \cite{Lieb-Ruskai-proof}, given by Aizenman and Cipolloni \cite{Aizenman-Cipolloni-proof}, (iii) obtaining a Lebesgue-type decomposition for positive operators by Ando \cite{Ando-Lebesgue}, etc.

Motivated by this idea, we define the NC parallel sum function on $\fD_2$ as
\begin{equation*}
    \fP(Z,W) := (I - Z)(2I - Z - W)^{-1}(I - W) = \big( (I - Z)^{-1} + (I - W)^{-1} \big)^{-1} = (I - W)(2I - Z - W)^{-1}(I - Z).
\end{equation*}
In this subsection, we will show that $\fP$ is a contractive stable NC rational function on $\fD_2$ that cannot be extended uniformly across the boundary, and establish its cyclicity using properties of \emph{accretive operators}. In fact, we shall prove a general result about cyclicity of stable \emph{accretive NC matrix functions}, i.e., $F \in M_k(H^\infty(\bD_Q))$ such that $\re F(X) \succeq 0$ for all $X \in \bD_Q$, and show that $\fP$ is accretive.

At the scalar level, we have
\begin{equation*}
    f(z,w) := \fP\vert_{\bD^2}(z,w) = \frac{(1-z)(1-w)}{2-z-w} = \left(\frac{1}{1-z} + \frac{1}{1-w} \right)^{-1} \foral (z,w) \in \bD^2.
\end{equation*}
This function appeared in the context of stability of digital filters in \cite{Goodman-scalar-parallel-sum}, where it was established that $|f| < 1$ on $\bD^2$ as follows.
\begin{enumerate}
    \item If $z \in \bD$, then $\re \left(\tfrac{1}{1-z}\right) > \frac{1}{2}$.
    \item Thus, $\re \left( \tfrac{1}{1-z} + \frac{1}{1-w} \right) > 1$ for all $(z,w) \in \bD^2$.
    \item Take the inverse above and use the fact that $\re \alpha > 1 \implies |\alpha| > 1$.
\end{enumerate}
This idea can be essentially generalized to $\fP$ through known properties of accretive operators, so let us briefly recount the required facts.

\subsubsection*{\texorpdfstring{\textbf{Accretive operators and cyclicity}}{Accretive operators and cyclicity}} An operator $A \in \sB(\cH)$ on any Hilbert space $\cH$ is said to be accretive if $\re A \succeq 0$, i.e.,
\begin{equation*}
    \re \langle Av, v \rangle_\cH \geq 0 \foral v \in \cH.
\end{equation*}
$A$ is said to be \emph{$\delta$-accretive} (or \emph{strongly accretive}) if $\re A \succeq \delta I$ for some $\delta > 0$.

\begin{remark}
    One generally considers unbounded/densely-defined accretive operators, wherein it becomes important to highlight if the operator is maximal accretive in a sense. In the case when the operator is bounded, this distinction need not be established as such an operator is trivially maximal accretive.
\end{remark}

The following properties of accretive operators are well-known (see, e.g., \cite[Sections V.3.10-11]{Kato-book}).

\begin{proposition}\label{prop:accretive.op.properties}
    If $A \in \sB(\cH)$ is accretive, then the following hold.
    \begin{enumerate}
        \item If $A$ is invertible then $A^{-1}$ is accretive.
        \item $(A + \lambda I)^{-1}$ exists for all $\lambda > 0$. Moreover,
        \begin{enumerate}
            \item $(A + \lambda I)^{-1}$ and $A (A + \lambda I)^{-1} = (A + \lambda I)^{-1} A$ are accretive,
            \item $\|(A + \lambda I)^{-1}\| \leq \frac{1}{\lambda}$, and
            \item $\|A (A + \lambda I)^{-1}\| \leq 1$.
        \end{enumerate}
        \item If $A$ is $\delta$-accretive, then $\|A^{-1}\| \leq \frac{1}{\delta}$.
    \end{enumerate}
\end{proposition}

The important observation in the above proposition is that the norm bounds are independent of the dimension of $\cH$. An immediate consequence of these properties is the cyclicity of bounded accretive NC matrix functions on $\bD_Q$.

\begin{corollary}\label{cor:cyclicity.of.accretive.NC.func}
    If $F \in M_k(H^\infty(\bD_Q))$ is $Q$-stable and accretive, then $F$ is left/right cyclic.
\end{corollary}

\begin{proof}
    Since $F$ is accretive, Proposition \ref{prop:accretive.op.properties}$(2)$ shows that the sequence
    \begin{equation*}
        \{ F_\lambda := (F + \lambda I)^{-1} \ : \ \lambda > 0 \}
    \end{equation*}
    consists of matrices of bounded NC maps. In fact, for all $\lambda > 0$ we have
    \begin{equation*}
        \|F_\lambda \|_Q \leq \frac{1}{\lambda} \qand \|F F_\lambda\|_Q = \|F_\lambda F\|_Q \leq 1.
    \end{equation*}
    We know from Proposition \ref{prop:weak-*.topology.properties}$(2)$ that weak-* convergence coincides with pointwise convergence on bounded sets. Thus, it is clear from the above equation and the $Q$-stability of $F$ that $F_\lambda F \xrightarrow{w^*} I$ and $F F_\lambda \xrightarrow{w^*} I$. Thus, $F$ is left/right cyclic by Lemma \ref{lem:cyclic.iff.identity}.
\end{proof}

\subsubsection*{\texorpdfstring{\textbf{Properties of the NC parallel sum}}{Properties of the NC parallel sum}} As discussed in the introduction to this subsection, we now replicate Goodman's \cite{Goodman-scalar-parallel-sum} argument for the scalar parallel sum function to show that $\fP$ is contractive, and use Corollary \ref{cor:cyclicity.of.accretive.NC.func} to prove its cyclicity.

\begin{corollary}\label{cor:cyc.of.NC.par.sum}
    The NC parallel sum $\fP$ is a contractive stable accretive NC rational function on $\fD_2$ that cannot be extended uniformly across the boundary. Thus, it is left/right cyclic.
\end{corollary}

\begin{proof}
    Note that $(I - Z)^{-1}$ is $\frac{1}{2}$-accretive for any $Z \in \fD_1$. Indeed, we have
    \begin{align*}
        (I - Z^*)^{-1} + (I - Z)^{-1} - I &= (I - Z^*)^{-1} \big( (I - Z) + (I - Z^*) - (I - Z^*)(I - Z) \big) (I - Z)^{-1} \\
        &= (I - Z^*)^{-1} (I - Z^* Z) (I - Z)^{-1} \\
        &\succeq 0.
    \end{align*}
    It then follows that
    \begin{equation*}
        \re\big((I - Z)^{-1} + (I - W)^{-1}\big) \succeq I \foral (Z,W) \in \fD_2,
    \end{equation*}
    i.e., $\fP^{-1}$ is $1$-accretive. Proposition \ref{prop:accretive.op.properties}$(3)$ shows at once that $\|\fP\|_Q \leq 1$ and so $\fP$ is contractive. That $\fP$ is $\fD_2$-stable is immediate from the definition, and that it is accretive follows from Proposition \ref{prop:accretive.op.properties}$(1)$. That $\fP$ is not uniformly continuous across the boundary is immediate by noting that, even at the scalar level, we have a singularity of the second kind at the boundary point $(1,1)$. Cyclicity of $\fP$ then follows from Corollary \ref{cor:cyclicity.of.accretive.NC.func}.
\end{proof}

\begin{remark}\label{rem:failure.NGN.par.sum.cyc}
    Note that the singularity at $(I,I)$ prevents us from using Theorem \ref{mainthm:cyc.NC.rationals} to determine the cyclicity of $\fP$. It might also be tempting to apply the NGN-type inequality for the atoms $I - Z$ and $I - W$ to somehow argue as in Theorem \ref{mainthm:cyc.matrix.free.poly} that $\fP$ is left/right cyclic. That is to say, we note that
    \begin{equation*}
        (I - rX)^{-1} \fP(X,Y) \to (2I - X - Y)^{-1} (I - Y)
    \end{equation*}
    for each $(X,Y) \in \fD_2$ as $r \to 1^-$. In particular, one may wish to employ an NGN-type bound for the $I - Z$ factor and show that the convergence above is weak-* convergence, so that
    \begin{equation*}
        [\fP]_L \supseteq [(2I - Z - W)^{-1} (I - W)]_L \supseteq [I - W]_L = H^\infty(\fD_2),
    \end{equation*}
    as $I - W$ is cyclic. This clearly does not work, however, because $(2I - Z - W)^{-1} (I-W)$ is unbounded on $\fD_2$.
\end{remark}

\begin{example}[The decoupled NC parallel sums]\label{eg:decoupled.NC.par.sum}
    It is interesting to note that the NC parallel sum function is symmetric with respect to the variables by definition and so, in a way, the inputs are `coupled' together. One way to decouple these variables is to consider
    \begin{align*}
        \fP_L(Z,W) := (2I - Z - W)^{-1} (I - Z)(I - W), \\
        \fP_R(Z,W) := (I - Z)(I - W) (2I - Z - W)^{-1}.
    \end{align*}
    It is straightforward to verify in this case that, in general, $\fP_L(Z,W) \neq \fP_L(W,Z)$ and $\fP_R(Z,W) \neq \fP_R(W,Z)$ on $\fD_2$. Moreover, an argument similar to the one given in Remark \ref{rem:failure.NGN.par.sum.cyc} would immediately yield that $\fP_L$ and $\fP_R$ are left and right cyclic, respectively, using Theorem \ref{mainthm:cyc.matrix.free.poly}. However, neither of these decoupled NC parallel sums lie in $H^\infty(\fD_2)$; simply test $\|\fP_L(X_t, Y_t)\|$ and $\|\fP_R(X_t, Y_t)\|$ as $t \to 0^+$ for
    \begin{equation*}
        X_t := \cos(t) \bmat{\cos(t) & \sin(t) \\ \sin(t) & -\cos(t)} \qand Y_t := \cos(t) \bmat{\cos(t) & -\sin(t) \\ -\sin(t) & -\cos(t)}.
    \end{equation*}
\end{example}

\section{Concluding remarks}

\subsection{NGN-type bounds using Fornasini--Marchesini realizations}\label{subsec:NGN.using.FM.realzn} The key step in the proof of Proposition \ref{prop:improved.bounds} is to use the fact that $P\sim_{\bU\bL} L_A$ to obtain an upper bound for $\big\|\big (P^{(r)} \big)^{-1}\big\|_Q$. Naturally, one can ask if there are other ways of bounding $\big\|\big (P^{(r)} \big)^{-1}\big\|_Q$ that lead to improved bounds in Proposition \ref{prop:improved.bounds}. In order to explore this question further, let $P\in \bC \langle Z \rangle$ be a $Q$-stable atom with $P(0)=I$, and let $(A,B,C,I)$ be a minimal \emph{Fornasini-Marchesini (FM) realization} of $P^{-1}$, so that $P(Z)^{-1}=I+C^*L_A(Z)^{-1}B(Z)$ (see, e.g., \cite[Section 5.2]{Helton-Klep-Volcic-Free-factor} and the references therein for background on FM-realizations).

It was pointed out to us by Jurij Vol\v{c}i\v{c} that, in this situation, $L_{A}$ will be irreducible and
\begin{equation}\label{eq:FM.factorization}
    \bmat{I & 0\\ C^*L_{A_0}(Z)^{-1} & I}\bmat{I& 0\\ 0& P(Z)}\bmat{L_{A_0}(Z) & B(Z)\\ 0 & I}=\bmat{I & B(Z)\\ 0 & I} \bmat{L_{A}(Z) & 0\\ 0 & I}\bmat{I & 0\\ C^* &  I},
\end{equation}
where $L_{A_0}$ is a linear pencil belonging to a certain minimal FM-realization of $P$. In particular, $L_{A_0}(X)$ is invertible for all $X\in \bM^d$.
By taking the inverse on both sides of  \eqref{eq:FM.factorization}, and using the upper (resp. lower) triangular form of the matrices to the left and right of $I\oplus P(Z)$, we get that $P^{-1}$ is equal to the $(2,2)$ entry of the inverse of the RHS of \eqref{eq:FM.factorization}. This allows us to obtain the following bound:
\begin{equation}\label{eq:FM.factorization.bound}
    \big\|\big (P^{(r)} \big)^{-1}\big\|_Q\leq \left\|\bmat{I & 0\\ -C^* &  I}\right\| \ \left\|\bmat{L_{A}(rZ)^{-1} & 0\\ 0 & I}\right\|_Q \ \left\|\bmat{I & -B(rZ)\\ 0 & I}\right\|_Q. 
\end{equation}
Next, using the above inequality and the fact that $L_A$ is irreducible, one can proceed exactly as in the proof of Proposition \ref{prop:improved.bounds}, to obtain upper bounds for $\sup_{r < 1} \big\| P \big(P^{(r)} \big)^{-1} \big\|_Q$ and $\sup_{r < 1} \big\|  \big(P^{(r)} \big)^{-1} P \big\|_Q$. However, in general this will not lead to an improvement of the bounds in Proposition \ref{prop:improved.bounds}, as the following example shows.

\begin{example}
    Consider the polynomial $P(Z,W)=I-\frac{ZW}{2}-\frac{WZ}{2}$ from Examples \ref{eg:Higman.linearization} and \ref{eg:concrete.bounds}(1). One can check that $(A,B,C,I)$, where $A=(A_Z,A_W)$ is as in Example \ref{eg:Higman.linearization}, $C=e_1\in \bC^3$ and
    \begin{equation*}
        B_Z=\bmat{0 & \frac{1}{\sqrt{2}} & 0}^t \qand B_W=\bmat{0& 0& \frac{1}{\sqrt{2}}}^t,
    \end{equation*}
    is an FM-realization of $P^{-1}$ (One can obtain this by using the linearization from Example \ref{eg:Higman.linearization} and arguing as in \cite[Remark 5.2]{Helton-Klep-Volcic-Free-factor}). Moreover, this realization is minimal since $A$ is irreducible by Example \ref{eg:Higman.linearization}. Next, it is readily checked that 
    \begin{equation*}
        \left\|\bmat{I & 0\\ -C^* &  I}\right\|= \frac{1+\sqrt{5}}{2} \qand \left\|\bmat{I & -B(Z,W)\\ 0 & I}\right\|_{\fD_2}\geq \left\|\bmat{I & -B(1,1)\\ 0 & I}\right\|=\frac{1+\sqrt{5}}{2}.
    \end{equation*}
    Moreover, by arguing exactly as in the beginning of the proof of Lemma \ref{lem:NGN.linear.pencil}, we get
    \begin{equation*}
        \left\|\bmat{L_{A}(rZ)^{-1} & 0\\ 0 & I}\right\|_{\fD_2}\leq \frac{1}{1-r}.
    \end{equation*}
    Therefore, for this minimal FM-realization of $P^{-1}$, the best possible bound one can extract from \eqref{eq:FM.factorization.bound} is
    \begin{equation*}
         \big\|\big (P^{(r)} \big)^{-1}\big\|_{\fD_2}\leq \frac{\big(1+\sqrt{5}\big)^2}{4(1-r)}. 
    \end{equation*}
    Given this inequality, we can proceed as at the end of the proof of Proposition \ref{prop:improved.bounds} to obtain
    \begin{equation*}
        \sup_{r < 1} \big\| (P^{(r)})^{-1} P \big\|_{\fD_2}\leq 1+\frac{\big(1+\sqrt{5}\big)^2}{2}\approx 6.24.
    \end{equation*}
    However, this is worse than the bound of $\leq 3$ that we obtained in Example \ref{eg:concrete.bounds}(1) by directly applying Proposition \ref{prop:improved.bounds} to $P$.
\end{example}

\subsection{Miscellaneous open questions} We close our discussion with some curious open problems that are suitable for further research.

\subsubsection*{\texorpdfstring{\textbf{Left vs. right cyclicity}}{Left vs. right cyclicity}} It was observed in \cite[Example 3.4]{JM-inner.e.g.} that a \emph{right inner} function $F \in H^\infty(\fB_d) \subset \bH^2_d$ may not necessarily be \emph{left inner}. Here, by left/right inner we mean a left/right isometric multiplier of $\bH^2_d$. Both cyclic and inner functions are important to understand the structure of these function spaces, so it is natural to ask in general:

\begin{problem}
    If $F \in M_k(H^\infty(\bD_Q))$ is, say, left cyclic then is it true that $F$ is also right cyclic?
\end{problem}

Our Theorem \ref{mainthm:cyc.NC.rationals} and Corollary \ref{cor:cyclicity.of.accretive.NC.func} provide situations where the answer to this problem is affirmative, but we do not currently have a counter-example nor a proof of this in general.

\subsubsection*{\texorpdfstring{\textbf{Non-cyclic stable NC rational functions}}{Non-cyclic stable NC rational functions}} As noted earlier, cyclicity of NC rational functions in $H^\infty(\fB_d)$ is equivalent to $\fB_d$-stability. Now, all of the examples and main theorems in this article consider $Q$-stable NC rational functions that turn out to be cyclic in $H^\infty(\bD_Q)$, but we do not know if this is always the case.

\begin{problem}
    Does there exist a $\bD_Q$-stable NC rational function that is not left/right cyclic in $H^\infty(\bD_Q)$? 
\end{problem}

\subsubsection*{\texorpdfstring{\textbf{The case of non-square matrix free polynomials}}{The case of non-square matrix free polynomials}} Lastly, it is clear that our technique of using a free NGN-type inequality relies on considering square matrix free polynomials, so we leave open the discussion on cyclicity in $M_{k \times l}(H^\infty(\bD_Q))$ for some $k, \ l \in \bN$ with $k \neq l$.

\section*{Acknowledgments}

We thank Robert T.W. Martin, Orr Moshe Shalit and Jurij Vol\v{c}i\v{c} for suggestions and helpful discussions on related topics.

%\nocite{*}
\bibliographystyle{plain}
\bibliography{bibliography}

%\printbibliography %~~~~~~~~ Use only if switching to Biber

\end{document}